\documentclass[12pt]{amsart}
\usepackage{graphicx}
\usepackage{amssymb}
\usepackage{amsfonts}
\usepackage{hyperref}
\usepackage{mathrsfs}
\usepackage[normalem]{ulem}
\numberwithin{equation}{section}
\newtheorem{theorem}{Theorem}[section]
\newtheorem{lemma}[theorem]{Lemma}
\newtheorem{corollary}[theorem]{Corollary}

\newtheorem{proposition}[theorem]{Proposition}
\theoremstyle{definition}

\newtheorem{remark}[theorem]{Remark}
\newtheorem{example}[theorem]{Example}

\theoremstyle{plain}

\numberwithin{figure}{section} 
\theoremstyle{plain}
\theoremstyle{plain}
\theoremstyle{remark}
\newtheorem*{acknowledgement*}{Acknowledgement}
\theoremstyle{example}

\usepackage{xcolor}

\newcommand{\cA}{{\mathcal A}}
\newcommand{\cB}{{\mathcal B}}
\newcommand{\cC}{{\mathcal C}}

\newcommand{\cE}{{\mathcal E}}

\newcommand{\cH}{{\mathcal H}}

\newcommand{\cR}{{\mathcal R}}

\newcommand{\cV}{{\mathcal V}}

\newcommand{\te}{{\theta}}

\newcommand{\ve}{{\varepsilon}}
\newcommand{\del}{{\delta}}

\newcommand{\gam}{{\gamma}}

\newcommand{\sig}{{\sigma}}
\newcommand{\al}{{\alpha}}

\newcommand{\bbE}{{\mathbb E}}

\newcommand{\bbN}{{\mathbb N}}
\newcommand{\bbP}{{\mathbb P}}
\newcommand{\bbR}{{\mathbb R}}
\newcommand{\bbT}{{\mathbb T}}
\newcommand{\bbZ}{{\mathbb Z}}
\newcommand{\bbI}{{\mathbb I}}

\def\fA{\mathfrak{A}}
\def\fg{\mathfrak{g}}
\def\fp{\mathfrak{p}}

\newcommand{\hj}{{\hat j}}
\newcommand{\hq}{{\hat  q}}

\newcommand{\bp}{{\mathbf{p}}}

\newcommand{\brX}{{\bar X}}

\newcommand{\brS}{{\bar S}}

\newcommand{\brj}{{\bar j}}
\newcommand{\brq}{{\bar q}}

\newcommand{\brphi}{{\bar\phi}}
\newcommand{\breps}{{\bar\varepsilon}}

\def\cR{{\mathcal{R}}}

\def\brs{{\bar s}}
\def\brR{{\bar R}}
\def\brkappa{{\bar\kappa}}

\def\fc{{\mathfrak{c}}}

\newcommand{\DS}{\displaystyle}

\begin{document}
\title[]{Edgeworth expansions for independent bounded integer valued random variables.}

 \vskip 0.1cm
\author{Dmitry Dolgopyat and Yeor Hafouta}


\dedicatory{  }

\maketitle

\begin{abstract}
We obtain asymptotic expansions for  local probabilities 
of partial sums for uniformly bounded independent but not necessarily identically distributed integer-valued random variables. The expansions involve products of polynomials and trigonometric
polynomials.
Our results do not require any additional assumptions.
As an application of our expansions we find necessary and sufficient conditions for the classical
Edgeworth expansion. It turns out that
there are two possible obstructions for the validity of
the Edgeworth expansion of order $r$. First, the distance between the distribution of the underlying partial sums modulo some $h\in \bbN$ and the uniform distribution could
fail to be $o(\sigma_N^{1-r})$, where $\sigma_N$ is the standard deviation of the partial sum. 
Second, this distribution
could have the required closeness but this closeness is unstable, in the sense that it could
be destroyed by removing finitely many terms. In the first case, the expansion of order $r$
fails. In the second case it may or may not hold depending on the behavior of the
derivatives of the characteristic functions of the summands 
whose removal causes the break-up of the uniform distribution.
We also show that a
quantitative version of the classical Prokhorov condition (for the strong local central limit theorem) 
is sufficient for Edgeworth expansions, and moreover this condition is, 
in some sense,  optimal.
\end{abstract}

\tableofcontents

\section{Introduction.}\label{SecInt}
Let $X_1,X_2,...$ be a  uniformly bounded sequence of independent integer-valued random variables.
Set $S_N=X_1+X_2+...+X_N$, $V_N=V(S_N)=\text{Var}(S_N)$ and $\sig_N=\sqrt{V_N}$. Assume also that $V_N\to\infty$ as $N\to\infty$.
Then the  central limit theorem (CLT) holds true, namely the distribution of $(S_N-\bbE(S_N))/\sig_N$ converges to the standard normal distribution as $N\to\infty$.

Recall that the local central limit theorem (LLT) states that, uniformly in $k$ we have
\[
\bbP(S_N=k)=\frac1{\sqrt{2\pi}\sig_N}e^{-\left(k-\bbE(S_N)\right)^2/2V_N}+o(\sig_N^{-1}).
\]
This theorem is also a classical result, and it has origins in  De Moivre-Laplace theorem.
The stable local central limit theorem (SLLT) states that the LLT holds true for any integer-valued square integrable independent sequence $X_1',X_2',...$ which differs from $X_1,X_2,...$ by a finite number of elements.  We recall a classical result due to Prokhorov.

\begin{theorem}
\cite{Prok} 
\label{ThProkhorov}
The SLLT holds iff
for each  integer $h>1$,
\begin{equation}\label{Prokhorov}
\sum_n \bbP(X_n\neq m_n \text{ mod } h)=\infty 
\end{equation}
where $m_n=m_n(h)$ is the most likely residue of $X_n$ modulo $h$.
\end{theorem}
We refer the readers' to \cite{Rozanov, VMT} for extensions of this result to the case when $X_n$'s are not necessarily bounded (for instance, the result holds true  when 
$\DS \sup_n\|X_n\|_{L^3}<\infty$). 
Related results for local convergence to more general 
limit laws are discussed in \cite{D-MD, MS74}.

The above result provides a necessary and sufficient condition for the SLLT. It turns out that the difference between LLT and SLLT is not that big.

\begin{proposition}
\label{PrLLT-SLLT}
Suppose $S_N$ obeys LLT. Then for each integer $h\geq2$ at least one of the following conditions occur:

either (a) $\DS \sum_n \bbP(X_n\neq m_n(h) \text{ mod } h)=\infty$.

or (b) $\exists j_1, j_2, \dots, j_k$ with $k<h$ such that 
$\DS \sum_{s=1}^k X_{j_s}$ mod $h$ is uniformly distributed.
In that case for all $N\geq \max(j_1, \dots, j_k)$ we have that 
$S_N$ mod $h$ is uniformly distributed. 
\end{proposition}

Since we could not find this result in the literature we include the proof in Section \ref{FirstOrder}. 

Next, we provide necessary and sufficient conditions for the regular LLT. 
We need an additional notation.
Let $\DS K=\sup_n\|X_n\|_{L^\infty}$.
Call $t$ \textit{resonant} if $t=\frac{2\pi l}{m}$ with $0<m\leq 2K$
and $0\leq l<m.$

\begin{theorem}
\label{ThLLT}
The following conditions are equivalent:

(a) $S_N$ satisfies LLT;

(b) For each $\xi\in \bbR\setminus \bbZ$, 
$\DS \lim_{N\to\infty} \bbE\left(e^{2\pi i \xi S_N}\right)=0$;

(c) For each non-zero resonant point $\xi$,
$\DS \lim_{N\to\infty} \bbE\left(e^{2\pi i \xi S_N}\right)=0$;

(d) For each integer $h$ the distribution of $S_N$ mod $h$ converges to uniform.
\end{theorem}

The proof of this result is also given in Section \ref{FirstOrder}. We refer the 
readers to 
\cite{Do, DS} for related results in more general settings.

The local limit theorem deals with approximation of $P(S_N=k)$ up to an error term of order $o(\sig_N^{-1})$. Given $r\geq1$, 
the Edgeworth expansion of order $r$ holds true 
if there are polynomials $P_{b, N}$, whose coefficients are uniformly bounded in $N$ and their degrees do no depend on $N,$ so that 
 uniformly in $k\in\bbZ$ we have
that
\begin{equation}\label{EdgeDef}
\bbP(S_N=k)=\sum_{b=1}^r \frac{P_{b, N} (k_N)}{\sigma_N^b}\fg(k_N)+o(\sigma_N^{-r})
\end{equation}
where $k_N=\left(k-\bbE(S_N)\right)/\sig_N$ 
and  $\fg(u)=\frac{1}{\sqrt{2\pi}} e^{-u^2/2}. $
In Section \ref{Sec5} we will show, in particular,
that Edgeworth expansions of any order $r$ are unique up to terms of order $o(\sig_N^{-r}),$ and so the case $r=1$  coincides with the LLT. 
Edgeworth expansions for discrete (lattice-valued) random variables have been studied in literature
for iid random variables  \cite[Theorem 4.5.4]{IL} 
\cite[Chapter VII]{Pet75}, (see also \cite[Theorem 5]{Ess}),
 homogeneous Markov chains  \cite[Theorems 2-4]{Nag2},
 decomposable statistics \cite{MRM}, or 
{ dynamical systems \cite{FL} with good spectral properties such as expanding maps.
Papers \cite{Bor16, GW17} discuss the rate of convergence in the LLT.
Results for non-lattice variables were obtained in \cite{Feller, BR76, CP, Br} (which considered random vectors) and \cite{FL} 
(see also \cite{Ha} for corresponding results for random expanding dynamical systems).

In this paper we 
obtain analogues of Theorems \ref{ThProkhorov} and \ref{ThLLT} for higher order Edgeworth expansions for independent but not identically distributed integer-valued uniformly bounded random variables.
We begin with the following result.

\begin{theorem}
\label{ThEdgeMN}
Let $\DS K=\sup_j\|X_j\|_{L^\infty}.$ 
For each $r\in\bbN$ there is a constant $R\!\!=\!\!R(r, K)$  such that 
the Edgeworth expansion of order $r$ holds~if 
\[
M_N:=\min_{2\leq h\leq 2K}\sum_{n=1}^N\bbP(X_n\neq m_n(h) \text{ mod } h)\geq R\ln V_N.
\]
In particular, $S_N$ obeys Edgeworth expansions of all orders if
$$ \lim_{N\to\infty} \frac{M_N}{\ln V_N}=\infty. $$
\end{theorem}
The number $R(r,K)$ can be chosen according to Remark \ref{R choice}.
This theorem is a quantitative version of Prokhorov's Theorem \ref{ThProkhorov}. 
We observe that logarithmic in $V_N$ growth of various non-perioidicity characteristics
of individual summands are often used in the theory of local limit theorems
(see e.g. \cite{Mal78, MS70, MPP}).  
We will see from the examples of Section \ref{ScExamples} that this result is close to optimal.
However, to  justify the optimality we need to understand the conditions necessary 
for the validity of the Edgeworth expansion.

\begin{theorem}\label{r Char}
For any $r\geq1$, the Edgeworth expansion of order $r$ holds if and only if for any nonzero resonant point $t$ and $0\leq\ell<r$ 
we have
\[
\bar \Phi_{N}^{(\ell)}(t)=o\left(\sig_N^{\ell+1-r}\right).
\]
where $\bar\Phi_{N}(x)=\bbE[e^{ix (S_N-\bbE(S_N))}]$
and $\bar\Phi_{N}^{(\ell)}(\cdot)$ is its $\ell$-th derivative. 
\end{theorem}

 This result\sout{s} generalizes Theorem \ref{ThLLT}, however in contrast with
that theorem,  in the case $r>1$ we also need to take into account the behavior of the derivatives 
of the characteristic function at nonzero resonant points. 
The values of the characteristic function at the resonant points $2\pi l/m$ have clear probabilistic meaning. Namely, they control the rate 
 equidistribution modulo $m$ (see part (d) of Theorem \ref{ThLLT} or Lemma \ref{LmUnifFourier}).
 Unfortunately, the probabilistic meaning of the derivatives is less clear, so it is desirable  to
characterize the validity of the Edgeworth expansions of orders higher than 1 without
 considering the derivatives. Example \ref{ExUniform} shows that this is impossible without additional assumptions. Some of the reasonable additional conditions are presented below.
 
We start with the expansion of order 2.

\begin{theorem}\label{Thm SLLT VS Ege}
Suppose $S_N$ obeys the SLLT. Then  the following are equivalent:

(a) Edgeworth expansion of order 2 holds;

(b) $|\Phi_N(t)|=o(\sigma_N^{-1})$ for each nonzero resonant point $t$;

(c) For each $h\leq 2K$ the distribution of $S_N$ mod $h$ is
$o(\sigma_N^{-1})$ close to uniform.
\end{theorem}
 Corollary \ref{CorNoDer} provides an extension 
 of Theorem \ref{Thm SLLT VS Ege} for expansions of an arbitrary order $r$ under an additional
 assumption that 
 $\DS \varphi:=\min_{t\in\cR}\inf_{n}|\phi_n(t)|>0$, where $\cR$ is the set of all nonzero resonant points. The latter condition  implies in particular, 
 that for each $\ell$ there is a uniform lower bound on the distance between the distribution of $X_{n_1}+X_{n_2}+\dots +X_{n_\ell}$ 
 \text{mod }$m$ and the uniform distribution, when $\{n_1, n_2,\dots ,n_\ell\} 
 \in\bbN^\ell$ and $m\geq2$.

 Next we discuss an analogue of Theorem \ref{ThProkhorov}  for expansions of order higher than 2. It requires
 a stronger condition  
which uses
an additional notation. Given
$j_1, j_2,\dots,j_s$ with $j_l\in [1, N]$ we write
$$ S_{N; j_1, j_2, \dots,j_s}=S_N-\sum_{l=1}^s X_{j_l}. $$
Thus $S_{N; j_1, j_2, \dots ,j_s}$ is a partial sum of our sequence with 
$s$ terms removed. We will say that $\{X_n\}$ 
{\em admits an Edgeworth expansion of order $r$ in a superstable way} 
(which will be denoted by $\{X_n\}\in EeSs(r)$) if for each $\brs$  and each
sequence $j_1^N, j_2^N,\dots ,j_{s_N}^N$ with $s_N\leq \brs$
there are polynomials $P_{b, N}$ whose coefficients are $O(1)$
in $N$ and their degrees do not depend on $N$ so that 
 uniformly in $k\in\bbZ$ we have
that
\begin{equation}\label{EdgeDefSS}
\bbP(S_{N; j_1^N, j_2^N, \dots, j_{s_N}^N}=k)=\sum_{b=1}^r \frac{P_{b, N} (k_N)}{\sigma_N^b}
\fg(k_N)+o(\sigma_N^{-r})
\end{equation}
and the estimates in $O(1)$ and $o(\sigma_N^{-r})$ are uniform in the choice
of the tuples $j_1^N, \dots ,j_{s_N}^N.$
That is, by removing a finite number of terms we can not destroy the validity of 
the Edgeworth expansion (even though the coefficients of the underlying polynomials
will of course depend on the choice of the removed terms). 
Let $\Phi_{N; j_1, j_2,\dots, j_s}(t)$ be the characteristic function of
$S_{N; j_1, j_2,\dots, j_s}.$

\begin{remark}
We note that in contrast with SLLT, in the definition of the superstrong Edgeworth
expansion one is only allowed to remove old terms, but not to add new ones. 
This difference in the definition is not essential, since adding terms with sufficiently
many moments (in particular, adding bounded terms) does not destroy the validity 
of the Edgeworth expansion. See  the proof of Theorem \ref{Thm Stable Cond} (i) or
the second part of Example \ref{ExNonAr}, starting with
equation \eqref{Convolve},
for details.
\end{remark}

\begin{theorem}\label{Thm Stable Cond}
(1) $S_N\in EeSs(1)$ (that is, $S_N$ satisfies the LLT 
in a superstable way) if and if it satisfies the SLLT.

(2) For arbitrary $r\geq 1$ the following conditions are equivalent:

(a) $\{X_n\}\in EeSs(r)$;

(b) For each $j_1^N, j_2^N,\dots ,j_{s_N}^N$ and each nonzero resonant point $t$ we have
$\Phi_{N; j_1^N, j_2^N,\dots, j_{s_N}^N}(t)=o(\sigma_N^{1-r});$

(c) For each $j_1^N, j_2^N,\dots ,j_{s_N}^N$, and each $h \leq 2K$ 
the distribution of $S_{N; j_1^N, j_2^N,\dots, j_{s_N}^N}$ mod $h$ is
$o(\sigma_N^{1-r})$ close to uniform.
\end{theorem}
\smallskip

To prove the above results we will show that 
for any order $r$, we can always approximate $\mathbb{P}(S_N=k)$ up to an error
$o(\sigma_N^{-r})$ provided that
instead of polynomials we use products of regular and
 the trigonometric polynomials. Those products allow us} to take into
account possible oscillatory behavior of $P(S_N=k)$ when $k$ belongs
to different residues mod $h$, where $h$ is denominator of a {\em resonant frequency.}
When $M_N\geq R V_N$
for $R$ large enough,  the new expansion coincides with the usual Edgeworth expansions. We thus derive that the condition $M_N\geq R\ln V_N$ is
 in a certain sense optimal.

\section{Main result}\label{Main R}
Let $X_1,X_2,...$ be a sequence of independent integer-valued random variables. For each $N\in\bbN$ we set
 $\DS S_N=\sum_{n=1}^N X_n$ and $V_N=\text{Var}(S_N)$. We assume in this paper that 
 $\DS \lim_{N\to\infty}V_N=\infty$ and that
 there is a constant $K$ such that
$$\sup_n \|X_n\|_{L^\infty}\leq K.
$$
Denote  $\sigma_N=\sqrt{V_N}.$ For each positive integer $m$, let
$ q_n(m)$ be the second largest among
$$\DS \sum_{l\equiv j \text{ mod }m} \bbP(X_n=l)=\bbP(X_n\equiv j \text{ mod }m),\,j=1,2,...,m$$ and $j_n(m)$ be the corresponding residue class.
Set $$ M_N(m)=\sum_{n=1}^N q_n(m)\quad\text{and}\quad M_N=\min_m M_N(m).$$

\begin{theorem}\label{IntIndThm}
 There $\exists J=J(K)<\infty$ and polynomials $P_{a, b, N}$,
where $a\in 0, \dots ,J-1,$ $b\in \mathbb{N}$,
 with degrees depending only on $b$ but not on $a, K$ or on any other 
characteristic of $\{X_n\}$, such that
the coefficients of $P_{a,b, N}$ are uniformly bounded in $N$, 
and, for any $r\geq1$ uniformly in $k\in\bbZ$ we have
$$\bbP(S_N=k)-\sum_{a=0}^{J-1} \sum_{b=1}^r \frac{P_{a, b, N} ((k-a_N)/\sigma_N)}{\sigma_N^b}
\fg((k-a_N)/\sigma_N) e^{2\pi i a k/J} =o(\sigma_N^{-r})
$$
where $a_N=\bbE(S_N)$ and $\fg(u)=\frac{1}{\sqrt{2\pi}} e^{-u^2/2}. $

Moreover, $P_{0,1,N}\equiv1$, 
and 
given $K, r$, there exists $R=R(K,r)$ such that if 
$M_N\geq R \ln V_N$ then we can choose $P_{a, b, N}=0$ for $a\neq 0.$ 
\end{theorem}
We refer the readers to (\ref{r=1})  for more details on these expansions in the case $r=1$, and to Section \ref{FirstOrder} for a discussion about the relations with local limit theorems. The resulting expansions in the case $r=2$ are given in (\ref{r=2'}). We note that the constants $J(K)$ and $R(K,r)$ can be recovered from  the proof of Theorem \ref{IntIndThm}.

\begin{remark}
Since the coefficients of the polynomials $P_{a,b,N}$ are uniformly bounded, the terms corresponding to $b=r+1$ are  of order  $O(\sig_N^{-(r+1)})$ uniformly in $k$. Therefore, in the $r$-th order expansion we actually get that the error term  is $O(\sig_N^{-(r+1)})$.
\end{remark}

\begin{remark}
In fact, the coefficients of the polynomials $P_{a,b,N}$ for $a>0$ are bounded by a constant times   $(1+M_N^{q})e^{-c_0  M_N}$, where $c_0>0$ depends only on $K$ and $q\geq 0$ depends only on $r$ and $K$. Therefore, these coefficient are small when $M_N$ is large. When $M_N\geq R(r,K)\ln V_N$ these coefficients  become of order $o(\sig^{-r}_N).$  Therefore, they only contribute to the error term, and so we can replace them by $0$, as stated in Theorem \ref{IntIndThm}.
\end{remark}

\begin{remark}
As in the derivation of the classical Edgeworth expansion, 
the main idea of the proof of Theorem \ref{IntIndThm} 
is the stationary phase analysis 
of the characteristic function. However, in contrast with the iid case 
 there may be
 resonances other than 0 which contribute to the oscillatory terms in the expansion.
Another interesting case where the classical Edgeworth analysis fails is the case of iid
terms where the summands are non-arithmetic but take only finitely many values.
It is shown in \cite{DF} that in that case, the leading correction to the Edgeworth expansion
also comes from resonances. However, in the case studied in \cite{DF} the geometry of 
resonances is more complicated, so in contrast 
to our Theorem \ref{IntIndThm}, 
\cite{DF} does not get the expansion of all orders.
\end{remark} 

\section{Edgeworth expansions under quantitative Prokhorov condition.}
\label{ScEdgeLogProkh}
In this section we prove Theorem \ref{ThEdgeMN}. In the course of the proof we obtain
the estimates of the characteristic function on  intervals 
not containing resonant points which will also 
play an important role in the proof of Theorem \ref{IntIndThm}.
The proof of Theorem \ref{IntIndThm} will be completed  in Section \ref{ScGEE}
where we analyze additional contribution coming from nonzero resonant points which appear in the case $M_N\leq R \ln \sig_N.$ Those contributions
constitute the source of the trigonometric polynomials in the generalized Edgeworth 
expansions.

\subsection{Characterstic function near 0.}
Here we recall some facts about the behavior of the characteristic function
near 0, which will be useful in the proofs of Theorems \ref{ThEdgeMN} and \ref{IntIndThm}.
The first result holds  for general uniformly bounded sequences 
$\{X_n\}$ (which are not necessarily integer-valued).
\begin{proposition}\label{PropEdg}
Suppose that $\DS \lim_{N\to\infty} \sig_N=\infty$, where $\sig_N\!=\!\sqrt{V_N}=
\sqrt{V(S_N)}$. Then
for $k=1,2,3,...$ there exists a sequence of polynomials $(A_{k,N})_N$ whose degree $d_k$ depends only on $k$ so that for any $r\geq 1$ there is $\del_r>0$ such that for all $N\geq1$ and $t\in[-\del_r\sig_N,\del_r\sig_N]$,
\begin{equation}\label{FinStep.0}
\bbE\left(e^{it(S_N-\bbE(S_N))/\sig_N}\right)=e^{-t^2/2}\left(1+\sum_{k=1}^r \frac{A_{k,N}(t)}{\sig_N^k}+\frac{t^{r+1}}{\sig_N^{r+1}} O(1)\right).
\end{equation}
Moreover, the coefficients of $A_{k,N}$ are algebraic combinations of moments of the $X_m$'s  
and they are uniformly bounded in $N$. Furthermore
\begin{equation}\label{A 1,n .1}
A_{1,N}(t)=
-\frac{i}{6}\gam_N t^3\,\,\text{ and }\,\,A_{2,N}(t)=\Lambda_{4}(\bar S_N)\sig_N^{-2}\frac{t^4}{4!}-\frac{1}{36}\gam_N^2 t^6
\end{equation}
where $\bar S_N=S_N-\bbE(S_N)$, $\gam_N=\bbE[(\bar S_N)^3]/\sig_N^2$ and $\Lambda_{4}(\bar S_N)$ is the fourth comulant of $\bar S_N$.
\end{proposition}
The proof  is quite standard, so we just sketch the argument.
The idea is to 
fix some $B_2>B_1>0$, and 
to  partition $\{1,...,N\}$ into intervals  $I_1,...,I_{m_N}$ 
so that
$\DS 
B_1\leq \text{Var}(S_{I_l})\leq B_2
$
where for each $l$ we set $\DS S_{I_l}=\sum_{j\in I_l}X_i$. 
It is clear that $m_N/\sig_N^2$ is 
bounded away from $0$ and $\infty$  uniformly in $N$.
Recall next that there are constants  $C_p$, $p\geq2$ so that for any $n\geq1$ and $m\geq 0$ we have
\begin{equation}
\label{CenterMoments}
\left\|\sum_{j=n}^{n+m}\big(X_j-\bbE(X_j)\big)\right\|_{L^p}\leq C_p\left(1+\left\|\sum_{j=n}^{n+m}\big(X_j-\bbE(X_j)\big)\right\|_{L^2}\right).
\end{equation}
This is a consequence of the multinomial theorem and some elementary estimates, and we refer the readers to either Lemma 2.7 in \cite{DS}, or Theorem 6.17 in \cite{Pel} for such a result in a much more general settings. 
Using the latter estimates we get that the $L^p$-norms of $S_{I_l}$ are uniformly bounded in $l$. This reduces the problem to the case when the variance of $X_n$ is uniformly bounded from below, and all the moments of $X_n-\bbE(X_n)$ are uniformly bounded. In this case, the proposition follows by considering the Taylor expansion of the function $\ln \bbE\big(e^{it(S_N-\bbE(S_N))/\sig_N}\big)+\frac12t^2$, 
see \cite[\S XVI.6]{Feller}.

\begin{proposition}
\label{PrHalf}
Given a square integrable random variable $X$, let $\bar{X}=X-\bbE(X).$ Then for each $h\in\mathbb{R}$ we have
$$\left|\bbE(e^{ih \bar X})-1\right|\leq  \frac12 h^2V(X).$$
\end{proposition}

\begin{proof}
Set $\varphi(h)=\bbE(e^{ih \bar X})$. Then by the integral form of the second order Taylor reminder we have 
$$
|\varphi(h)-\varphi(0)-h\varphi'(0)|=|\varphi(h)-\varphi(0)|=\left|\int_0^h(t-h)\varphi''(t)dt\right|$$
$\hskip4.7cm \DS \leq 
V(X)\int_{0}^{|h|}(|h|-t)dt= \frac12 h^2V(X).
$
\end{proof}

\subsection{ Non resonant intervals.}
\label{SSNonRes}
As in almost all the proofs of the LLT, the  starting point in the proof of Theorem \ref{ThEdgeMN} (and Theorem \ref{IntIndThm}) is that 
for  $k, N\in\bbN$ we have
\begin{equation}
\label{EqDual}
2\pi \bbP(S_N=k)=\int_{0}^{2\pi} e^{-itk}\bbE(e^{it S_N})dt.
\end{equation}

Denote $\bbT=\bbR/2\pi \bbZ.$ Let
$$ \Phi_N(t)=\bbE(e^{it S_N})=\prod_{n=1}^N \phi_n(t) \quad
\text{where}\quad \phi_n(t)=\bbE(e^{it X_n}).$$

Divide $\bbT$ into intervals $I_j$ of small size $\delta$
such that each interval contains at most one resonant point and this point is strictly inside
$I_j.$ We call an interval resonant if it contains a resonant point inside. Then
\begin{equation}\label{SplitInt}
2\pi \bbP(S_N=k)=\sum_{j}\int_{I_j} e^{-itk}\bbE(e^{it S_N})dt.
\end{equation}
We will consider  the integrals appearing in the above sum individually.

\begin{lemma}\label{Step1}
There are constants $C,c>0$ which depend only on $\del$ and $K$ 
so that for any non-resonant interval $I_j$ and $N\geq1$ we have
$$  \int_{I_j} |\Phi_N(t)| dt\leq C e^{-c V_N}. $$
\end{lemma}
\begin{proof}
Let $\hq_n, \brq_n$ be the largest and the second largest values of $\bbP(X_n=j)$
and let $\hj_n, \brj_n$ be the corresponding values.
Note that
\begin{equation}
\label{phiNSum}
 \phi_n(t)=\hq_n e^{it \hj_n}+\brq_n e^{it \brj_n}+\sum_{l\neq \hj_n, \brj_n}
\bbP(X_n=l) e^{i t l}. 
\end{equation}
Since $I_j$ is non resonant, the angle between $e^{it \hj_n}$ and 
$e^{it \brj_n}$ is uniformly bounded from below. Indeed if this was not the case we would have
$t \brj_n-t\hj_n\approx 2\pi l_n$ for some $l_n\in\bbZ.$ 
Then $t\approx \frac{2\pi l_n}{m_n}$ where $m_n=\brj_n-\hj_n$ contradicting  the assumption
that $I_j$ is non-resonant. Accordingly $\exists c_1>0$ such that
$\DS \left|e^{it \hj_n}+e^{it \brj_n}\right|\leq 2-c_1. $ Therefore 
$$ \left|\hq_n e^{it \hj_n}+\brq_n e^{it \brj_n}\right|
\leq { (\hq_n-\brq_n)+\brq_n \left| e^{it \hj_n}+e^{it \brj_n} \right|}
\leq \hq_n+\brq_n-2c_1 \brq_n. $$
Plugging this into \eqref{phiNSum}, we conclude that 
$|\phi_n(t)|\leq 1-2c_1 \brq_n$ for $t\in I_j.$ Multiplying these estimates  over
$n$ and using that $1-x\leq e^{-x},\, x>0$, we get
$$ |\Phi_N(t)|\leq e^{-2c_1\sum_n \brq_n}. $$
Since $V(X_n)\leq c_2 \brq_n$ for a suitable constant $c_2$ we can rewrite the preceding as
\begin{equation}\label{NonResDec}
 |\Phi_N(t)|\leq e^{-c_3V_N},\, c_3>0. 
\end{equation}
Integrating over $I_j$ we obtain the result.
\end{proof}

\subsection{Prokhorov estimates}
Next we consider the case where $I_j$ contains a nonzero resonant point $t_j=\frac{2\pi l}{m}.$

\begin{lemma}\label{Step2}
There is a constant $c_0$ which depends only on $K$ so that for any nonzero resonant point $t_j=2\pi l/m$ we have
\begin{equation}\label{Roz0}
\sup_{t\in I_j}|\bbE(e^{it S_N})|\leq e^{-c_0 M_N(m)}.
\end{equation}
Thus, for any $r\geq1$ there is a constant $R=R(r,K)$
such that if $M_N(m)\geq R \ln V_N$,  then the  integral  
$\int_{I_j} e^{-itk}\bbE(e^{it S_N})dt$  is $o(\sigma_N^{-r})$ uniformly in $k$,
and so it only contributes to the error term.
\end{lemma}

\begin{proof}
The estimate (\ref{Roz0}) follows from the arguments in \cite{Rozanov}, but 
for readers' convenience we recall its proof. Let $X$ be an integer-valued random variable so that $\|X\|_{L^\infty}\leq K$. Let $t_0=2\pi l/m$ be a nonzero resonant point, where $\gcd(l,m)=1$. Let $t\in\bbT$ be so that 
\begin{equation}
\label{Neart0}
|t-t_0|\leq\del, 
\end{equation}
where $\del$ is a small positive number. Let $\phi_X(\cdot)$ denote the characteristic function of $X$. Since $x\leq e^{x-1}$ for any real $x$ we have
\[
|\phi_X(t)|^2\leq e^{|\phi(t)|^2-1}.
\]
Next, we have 
\[
|\phi_X(t)|^2-1=\phi(t)\phi(-t)-1=\sum_{j=-2K}^{2K}\sum_s\tilde P_j
 \left[\cos(t_j)-1\right]
\]
where 
\[
\tilde P_j=\sum_{s}\bbP(X=s)\bbP(X=j+s).
\]
Fix some $-2K\leq j\leq 2K$. 
 We claim that if $\del$ in \eqref{Neart0} is small enough and $j\not\equiv 0\text{ mod }m$ then for each
integer $w$ we have $|t-2\pi w/j|\geq \ve_0$ for some $\ve_0>0$ which depends only on $K$. This follows from the fact that $-2K\leq j\leq 2K$ and that $2\pi w/j\not=t_0$ (and there is a finite number of resonant points). Therefore,
\[
\cos(tj)-1\leq -\del_0
\]
for some $\del_0>0$.
On the other hand, if $j=km$ for some integer $k$ then with $w=lk$ we have
\begin{eqnarray*}
\cos(tj)-1=-2\sin^2(tj/2)=-2\sin^2\left((tj-2\pi w)/2\right)\\=-2\sin^2\left(j(t-t_0)/2\right)\leq 
-\del_1(t-t_0)^2
\end{eqnarray*}
for some $\del_1>0$ (assuming that $|t-t_0|$ is small enough). We conclude that 
\[
|\phi_X(t)|^2-1\leq -\del_0\sum_{j\in A}\tilde P_j-\del_1(t-t_0)^2\sum_{j\in B}\tilde P_j
\]
where $A=A(X)$ is the set of $j$'s between $-2K$ and $2K$ so that $j\not\equiv 0\text{ mod }m$  and $B=B(X)$ is its complement in $\bbZ\cap[-2K,2K]$. Let $s_0$ be the most likely 
 residue  of $X$  mod $m$ and $s_1$ be the second
most likely residue class.
Since 
$$ \bbP(X\equiv s_0 \text{ mod }m)\geq \frac{1}{m}
\quad\text{and}\quad 
\bbP(X\equiv s_1 \text{ mod }m)=q_m(X) $$
it follows that
$\DS \sum_{j\in A} \tilde P_j\geq \frac{q_m(X)}{m}.$

Combining this with the trivial bound 
$\DS \sum_{j\in B} \tilde P_j\geq \bbP^2(X\equiv s_0)\geq \frac{1}{m^2}$ we obtain
$$
|\phi_X(t)|\leq \exp-\left[\frac12\left(\frac{\del_0 q_m(X)}{m} +
\frac{\del_1 (t-t_0)^2}{m^2}\right)\right].
$$
Applying the above with $t_0=t_j$ and $X=X_n$, $1\leq n\leq N$ we get that 
\begin{equation}\label{CharEstRoz}
|\Phi_N(t)|\leq e^{-c_0M_N(m)-\bar{c}_0 N(t-t_j)^2 }\leq e^{-c_0M_N(m)}
\end{equation}
where $c_0$ is some constant.
\end{proof}

\begin{remark}
Using the first inequality in (\ref{CharEstRoz}) and 
arguing as in \cite[page 264]{Rozanov}, 
we can deduce that
there are positive constants $C, c_1, c_2$ such that
\begin{equation}\label{RozArg}
\int_{I_j}|\bbE(e^{it S_N})|dt\leq C\left(e^{-c_1\sig_N}+\frac{e^{-c_2M_N(m)}}{\sig_N}\right).
\end{equation}
This estimate plays an important role in the proof of the SLLT in \cite{Rozanov}, but for our purposes a weaker bound
\eqref{Roz0} is enough. Note also that in order to prove (\ref{Roz0}) we could have just used the trivial inequality $\cos(t_j)-1\leq 0$ when $j\equiv 0\text{ mod }m$, but we have decided to present this part from \cite{Rozanov} in full.
\end{remark}

\begin{remark}\label{R choice}
Let $d_{\cR}$ be the minimal distance between two different resonant points. Then, when $\del<2d_{\cR}$, we can take $\del_0=1-\cos(d_{\cR})$ in the proof of Lemma \ref{Step2}. Therefore, we can take $c_0=\frac{1-\cos(d_{\cR})}{4K}$ in \eqref{Roz0}. 	Hence
Lemma \ref{Step2} holds with
 $R(r,K)=\frac{r+1}{2c_0}.$
\end{remark}
\vskip0.2cm

\subsection{ Proof of Theorem \ref{ThEdgeMN}} 
\label{Cmplt1}
Fix some $r\geq1$.
Lemmas \ref{Step1} and \ref{Step2} show that if $M_N\geq R(r,K)\ln V_N$,  
then all the integrals in the right hand side of (\ref{SplitInt}) are of order $o(\sig_N^{-r})$, except for the one corresponding to the resonant point $t_j=0$. That is, for any $\del>0$ small enough, uniformly in $k$ we have
$$2\pi \bbP(S_N=k)=\int_{-\del}^\del e^{-ih k}\Phi_N(h)dh+o(\sig_N^{-r}).
$$
In order to complete the proof of Theorem \ref{ThEdgeMN}, we need to expand the above integral. Making a change of variables $h\to h/\sig_N$ and using Proposition \ref{PropEdg}, 
we conclude that if $\del$ is small enough then
$$
\int_{-\del}^\del e^{-ih k}\Phi_N(h)dh=$$
$$\sig_N^{-1}\int_{-\del\sig_N}^{\del\sig_N}e^{-ihk_N}e^{-h^2/2}\left(1+\sum_{u=1}^r \frac{A_{u,N}(h)}{\sig_N^k}+\frac{h^{r+1}}{\sig_N^{r+1}} O(1)\right)dh
$$
where $k_N=\left(k-\bbE(S_N)\right)/\sig_N$. Since the coefficients of the polynomials $A_{u,N}$ are uniformly bounded in $N$, we can just replace the above integral with the corresponding integral over all $\bbR$ (i.e. replace $\pm\del\sig_N$ with $\pm\infty$). Now the Edgeworth expansions are
achieved using that for any nonnegative integer $q$ we have that $(it)^qe^{-t^2/2}$ is the Fourier transform of the $q$-th derivative of $\textbf{n}(t)=\frac{1}{\sqrt{2\pi}}e^{-t^2/2}$ and that for any real $a$,
\begin{equation}\label{Fourir}
\int_{-\infty}^\infty e^{-iat}\widehat{\textbf{n}^{(q)}}(t) dt=\textbf{n}^{(q)}(a)=\frac{1}{\sqrt{2\pi}}(-1)^{q}H_q(a)e^{-a^2/2}
\end{equation}
where $H_q(a)$ is the $q$-th Hermite polynomial.

\section{Generalized Edgeworth expansions: Proof of Theorem Theorem~ \ref{IntIndThm}}
\label{ScGEE}
\subsection{Contributions of resonant intervals.}
Let $r\geq1$. As in the proof of Theorem \ref{ThEdgeMN}, our starting point is the equality
\begin{equation}
\label{EqDual}
2\pi \bbP(S_N=k)=\int_{0}^{2\pi} e^{-itk}\bbE(e^{it S_N})dt=\sum_{j}\int_{I_j} e^{-itk}\bbE(e^{it S_N})dt
\end{equation}
which holds for any $k\in\bbN$. We will consider  the integrals appearing in the above sum individually.
By Lemma \ref{Step1}  the integrals over non-resonant intervals are of order $o(\sig_N^{-r})$, and so they can be disregarded. Moreover, in \S \ref{Cmplt1} we have expanded the integral over the resonant interval containing $0$.
 Now we will see that in the case $M_N< R(r,K) \ln V_N$ 
the contribution of nonzero resonant points need not be negligible.

Let $t_j=\frac{2\pi l}{m}$ be a nonzero resonant point so that  $M_N(m)<R(r,K) \ln V_N$ and let $I_j$ be the resonant interval containing it.  Theorem~\ref{IntIndThm} will follow from an appropriate expansion of the integral
$$\int_{I_j} e^{-itk}\bbE(e^{it S_N})dt.$$
We need the following simple result, which for readers' convenience is formulated as a lemma.
\begin{lemma}\label{EpsLem}
There exists $\breps>0$ so that for each $n\geq1$ with $q_n(m)\leq \breps$ we have  $|\phi_n(t_j)|\geq\frac12$. In fact, we can take $\breps=\frac1{4m}$.
\end{lemma}

\begin{proof}
Recall that $t_j=2\pi l/m$. The lemma follows since for any random variable $X$ we have
$\DS |\bbE(e^{i t_j X})|=$
$$\left|e^{it_js(m,X)}-\sum_{u\not\equiv s(m,X)\text{ mod m}}\big(e^{it_j s(m,X)}-e^{it_j u}\big)P(X\equiv u\text{ mod } m)\right|
$$
$$
\geq 1-2mq(m,X)
$$
where $s(m,X)$ is the most likely value of $X\text{ mod }m$ and
$q(m,X)$ is the second largest value among $P(X\equiv u\text{ mod } m)$, $u=0,1,2,...,m-1$. Therefore, we can take  $\breps=\frac1{4m}$.
\end{proof}
Next, set $\breps=\frac{1}{8K}$ and let $N_0=N_0(N,t_j,\breps)$ be the number of all $n$'s between $1$ to $N$ so that $q_n(m)>\breps$.
Then $N_0\leq \frac{R \ln V_N}{\breps}$ because $M_N(m)\leq R \ln V_N.$
By permuting the indexes $n=1,2,...,N$ if necessary we can assume that $q_n(m)$ is non increasing.
Let $N_0$ be the largest number such that $q_{N_0}\geq \breps.$ 
  Decompose
\begin{equation}
\label{NPert-Pert}
\Phi_N(t)=\Phi_{N_0}(t) \Phi_{N_0, N}(t) 
\end{equation}
where
$\DS \Phi_{N_0, N}(t)=\prod_{n=N_0+1}^N \phi_n(t).$

\begin{lemma}\label{Step3}
If the length $\del$ of $I_j$ is small enough then for any $t=t_j+h\in I_j$ and $N\geq1$ we have
\[ 
\Phi_{N_0,N}(t)=\Phi_{N_0,N}(t_j)\Phi_{N_0,N}(h)  \Psi_{N_0,N}(h)
\]
 where
$$ \Psi_{N_0,N}(h)=\exp\left[O(M_N(m))\sum_{u=1}^\infty (O(1))^u h^u\right]. $$
\end{lemma}
\begin{proof}
 Denote
$$ \mu_n=\bbE(X_n), \quad \brX_n=X_n-\mu_n, \quad \brphi_n(t)=\bbE(e^{it \brX_n}).$$ 
Let $j_n(m)$ be the most likely residue mod $m$ for $X_n.$ Decompose
$$ \brX_n=s_n+Y_n+Z_n$$ where $Z_n\in m\bbZ$,  $s_n=j_n(m)-\mu_n$, so that 
$\bbP(Y_n\neq 0)\leq m q_n(m).$ Then for $t=t_j+h$,
\begin{equation}
\label{IndCharTaylor}
 \brphi_n(t)=e^{i t_j s_n} \bbE\left(e^{i t_j Y_n} e^{ih \brX_n}\right)=\brphi_n(t_j)  \psi_n(h)
 \end{equation}
 where 
$$  \psi_n(h)=
\left(1+
\frac{i h \bbE(e^{i t_j Y} \brX_n)-\frac{h^2}{2} \bbE(e^{i t_j Y} (\brX_n)^2)+\dots}{\bbE(e^{i t_j Y_n})}\right) . $$
Next, using that for any $x\in(-1,1)$ we have 
$$1+x=e^{\ln (1+x)}=e^{x-x^2/2+x^3/3-...}$$ 
we obtain that for $h$ close enough to $0$,
\begin{equation}\label{Expand}
\psi_n(h)=
\exp\left(\sum_{k=1}^\infty\frac{(-1)^{k+1}}{k}\left(\frac1{\bbE(e^{it_jY_n})}\sum_{q=1}^\infty \frac{(ih)^q}{q!}\bbE(e^{it_j Y_n}(\bar X_n)^q)\right)^k\right)
\end{equation}
\begin{eqnarray}
=\exp\left(\sum_{k=1}^\infty\frac{(-1)^{k+1}}{k}\sum_{1\leq j_1,...,j_k}\frac1{(\bbE(e^{it_jY_n}))^k}\prod_{r=1}^{k}\frac{(ih)^{j_r}}{j_r!}\bbE(e^{it_j Y_n}(\bar X_n)^{j_r})\right)\nonumber\\=
\exp\left(\sum_{u=1}^\infty\left(\sum_{k=1}^{u}\frac{(-1)^{k+1}}{k}\sum_{j_1+...+j_k=u}\,\prod_{r=1}^{k}\frac{\bbE(e^{it_j Y_n}(\bar X_n)^{j_r})}{\bbE(e^{it_j Y_n})j_r!}\right)(ih)^u\right).\nonumber
\end{eqnarray}
Observe next that 
\[
\bbE[e^{it_j Y_n}(\bar X_n)^{j_r}]=\bbE\left[(e^{it_j Y_n}-1)\big((\bar X_n)^{j_r}-\bbE[(\bar X_n)^{j_r}]\big)\right]+
\bbE[(\bar X_n)^{j_r}]\bbE(e^{it_j Y_n})
\]
and so with $C=2K$, we have
\[
\frac{\bbE[e^{it_j Y_n}(\bar X_n)^{j_r}]}{\bbE(e^{it_j Y_n})}=O(q_n(m))O(C^{j_r})+\bbE[(\bar X_n)^{j_r}].
\]
Plugging this into (\ref{Expand}) and using
 that for 
 $h$ small enough,
\begin{equation*}
\exp\left[
\sum_{u=1}^\infty\left(\sum_{k=1}^u\frac{(-1)^{k+1}}k\sum_{j_1+...+j_k=u}\prod_{r=1}^k\frac{\bbE(\brX_n^{j_r})}{j_r!}\right)(ih)^u\right]=\bbE\left(e^{ih \brX_n}\right)
\end{equation*}
we conclude that 
$$
 \psi_n(h)=\bbE(e^{ih\bar X_n})
\exp\left[\sum_{u=1}^\infty(O(1))^uO(q_n(m))h^u\right].$$
Therefore, 
\[ 
\Phi_{N_0,N}(t)=\Phi_{N_0,N}(t_j)\Phi_{N_0,N}(h)  \Psi_{N_0,N}(h)
\]
 where
$\DS \Psi_{ N_0, N}(h)=\exp\left[O(M_N(m))\sum_{u=1}^\infty (O(1))^u h^u\right]. $
\end{proof}
\begin{remark}\label{CoeffRem}
 We will see in \S\ref{Fin} that the coefficients of the polynomials appearing in Theorem~\ref{IntIndThm}
depend on the coefficients of the power series $\Psi_{N_0, N}(h)$ 
(see,  in particular, \eqref{Cj(k)}).
The first term in this  series is $\DS ih \sum_{n={N_0+1}}^N a_{n,j}$, where 
\begin{equation}\label{a n,j}
a_{n,j}=\frac{\bbE[(e^{it_j Y_n}-1)\bar X_n]}{\bbE(e^{i t_j Y_n})}=\frac{\bbE(e^{it_j X_n}\bar X_n)}{\bbE(e^{it_j X_n})}
\end{equation}
while the second term is $\DS \frac{h^2}{2} \sum_{n={N_0+1}}^N b_{n,j}$, where 
\begin{equation}
\label{SecondTerm}
 b_{n,j}=\frac{\bbE[(e^{it_j Y_n}-1)\bar X_n]^2}{\bbE(e^{i t_j Y_n})^2}-\frac{\bbE[(e^{it_j Y_n}-1)((\bar X_n)^2-V(X_n))]}{\bbE(e^{i t_j Y_n})}
\end{equation}
$$
=a_{n,j}^2-\frac{\bbE\big(e^{it_j X_n}(\bar X_n)^2\big)}{\bbE(e^{it_j X_n})}.
$$
In Section \ref{Sec2nd} we will use (\ref{a n,j}) to compute the coefficients of the polynomials from Theorem \ref{IntIndThm} in the case $r=2$, and \eqref{SecondTerm}
 is one of the main ingredients for the computation in the case $r=3$ 
 (which will not be explicitly discussed in this manuscript).
\end{remark}

The next step in the proof of Theorem \ref{IntIndThm} is the following.
\begin{lemma}\label{Step4}
For $t=t_j+h\in I_j$ we can decompose
\begin{equation}\label{S4}
\Phi_{N_0}(t)=\Phi_{N_0}(t_j+h)=\sum_{l=0}^L \frac{\Phi_{N_0}^{(l)}(t_j)}{l!} h^l+O\left((h \ln V_N)^{L+1}\right).
\end{equation}
\end{lemma}
\begin{proof}
The lemma follows from the observation that 
 the derivatives of $\Phi_{N_0}$ satisfy
$|\Phi_{N_0}^{(k)}(t)|\leq O(N_0^k)\leq (C \ln V_N)^k$.
\end{proof}

\subsection{Completing the proof}\label{Fin}
 Recall \eqref{EqDual} and  consider a resonant interval $I_j$ which does not contain $0$ such that $M_N(m)\leq R\ln\sig_N$. Set $U_j=[-u_j,v_j]=I_j-t_j$. Let $N_0$ be as described below Lemma \ref{EpsLem}.
Denote
\begin{equation}
\label{SNN0}
S_{N_0,N}=S_N-S_{N_0},\,  S_0=0, 
\end{equation}
$$V_{N_0,N}=\text{Var}(S_N-S_{N_0})=V_N-V_{N_0}\quad\text{and}
\quad \sig_{N_0,N}=\sqrt{V_{N_0,N}}.$$
Then
\begin{equation}\label{Vars}
V_{N_0,N}=V_N+O(\ln V_N)=V_N(1+o(1)).
\end{equation}

Denote $h_{N_0,N}=h/\sig_{N_0,N}.$
By \eqref{FinStep.0},   if $|h_{N, N_0}|$  
is small enough then
\begin{equation}\label{FinStep}
\bbE(e^{ih_{N_0,N} S_{N_0,N}})=
\end{equation}
$$e^{ih_{N_0,N}\bbE(S_{N_0,N})}e^{-h^2/2}\left(1+\sum_{k=1}^r \frac{A_{k,N_0,N}(h)}{\sig_{N_0,N}^k}+\frac{h^{r+1}}{\sig_{N_0,N}^{r+1}} O(1)\right)$$
where $A_{k,N_0,N}$ are polynomials with bounded coefficients (the degree of $A_{k,N_0,N}$ depends only on  $k$).
Let us now evaluate
$\DS
\int_{I_j}e^{-itk}\bbE(e^{it S_N})dt.
$
By Lemma \ref{Step3},
\begin{equation}
\label{ResInt}
\int_{I_j}e^{-itk}\Phi_N(t)dt=
\end{equation}
$$ e^{-it_j k}\Phi_{N_0,N}(t_j)\int_{U_j}e^{-ihk}\Phi_{N_0}(t_j+h)\Phi_{N_0,N}(h)  \Psi_{N_0, N}(h)\; dh.
$$

Therefore, it is enough to expand the integral on the RHS of \eqref{ResInt}. 
Fix a large positive integer $L$ and plug \eqref{S4} into \eqref{ResInt}. 
Note that
for $N$ is large enough, $h_0$  small enough and $|h|\leq h_0$, 
Proposition \ref{PrHalf} and \eqref{Vars} show that there
exist positive constants $c_0, c$ such that
\begin{equation}\label{expo}
 |\Phi_{N_0,N}(h)|=|\bbE(e^{ih S_{N_0,N}})|\leq e^{-c_0(V_N-V_{N_0})h^2}\leq
 e^{-cV_N h^2}.
\end{equation} 
 Thus, the contribution coming from the term $O\left((h \ln V_N)^{L+1}\right)$ in the right hand side of (\ref{S4}) is at most of order 
\[
V_N^{R\del}(\ln V_n)^{L+1}\int_{-\infty}^\infty h^{L+1}e^{-c V_N h^2}dh
\] 
where $\del$ is the diameter of $I_j$. Changing variables $x=\sig_N h$, where $\sig_N=\sqrt{V_N}$ we get that  the latter term is  of order $(\ln V_n)^{L+1}\sig_{N}^{-(L+1-2R\del)}$ and so when $L$ is large enough we get that this term is $o(\sig_N^{-r-1})$ (alternatively, we can take $L=r$ and $\del$ to be sufficiently small). This means that it is enough to expand each integral of the form
\begin{equation}
\label{HLInt}
\int_{U_j}e^{-ih k}h^l\Phi_{N_0,N}(h)  \Psi_{N_0, N}(h) dh
\end{equation}
where $l=0,1,...,L$ (after changing variables the above integral is divided by $\sig_{N_0,N}^{l+1}$). Next, 
Lemma \ref{Step3} shows
that for any 
$d\in \bbZ$ 
we have
\begin{equation}
\label{EqOrderD}
 \Psi_{N_0, N}(h)
=1+\sum_{u=1}^{d}C_{w,N}h^u+h^{d+1}O(1+M_N(m)^{d+1})|V_N|^{O(|h|)},
\end{equation}
where  $C_{w,N}=C_{w,N, t_j}$ are $O(M_{N}^u(m))=O((\ln V_N)^{u})$. Note that, with $a_{n,j}$ and $b_{n,j}$ defined in Remark \ref{CoeffRem}, we have
\begin{equation}\label{C 1 N}
C_{1,N}=i\sum_{n=N_0+1}^{N}a_{n,j}
\end{equation}
and 
\[
C_{2,N}=\frac12\sum_{n=N_0+1}^{N}b_{n,j}-\frac 12\left(\sum_{n=N_0+1}^{N}a_{n,j}\right)^2.
\]

Take $d$ large enough and plug \eqref{EqOrderD} into \eqref{HLInt}.
Using (\ref{expo}), we get again that the contribution of the term 
$$h^{d+1}O(1+M_N(m)^{d+1})|V_N|^{O(|h|)}h^l\Phi_{N_0,N}(h)$$ to the above integral is $o(\sig_N^{-r})$. Thus, it is enough to expand each term of the form
\[
\int_{U_j}e^{-ih k}h^{q}\Phi_{N_0,N}(h)dh
\] 
where $0\leq q\leq L+d$.
Using (\ref{FinStep}) and making the  change of variables $h\to h/\sig_{N_0,N}$ it is enough to expand 
\begin{equation}\label{MainInt}
\int_{-\infty}^\infty e^{-ih (k-\bbE[S_{N_0,N}])/\sig_{N_0,N}}h^qe^{-h^2/2}\left(1+\sum_{w=1}^r \frac{A_{w,N_0,N}(h)}{\sig_{N_0,N}^w}+\frac{h^{r+1}}{\sig_{N_0,N}^{r+1}} O(1)\right)
 \frac{dh}{\sig_{N_0, N}}.
\end{equation}
This is achieved by using that $(it)^qe^{-t^2/2}$ is the Fourier transform of the $q$-th derivative of $\textbf{n}(t)=\frac{1}{\sqrt{2\pi}}e^{-t^2/2}$ and that for any real $a$,
\begin{equation}\label{Fourir}
\int_{-\infty}^\infty e^{-iat}\widehat{\textbf{n}^{(q)}}(t) dt=\textbf{n}^{(q)}(a)=\frac{1}{\sqrt{2\pi}}(-1)^{q}H_q(a)e^{-a^2/2}
\end{equation}
where $H_q(a)$ is the $q$-th Hermite polynomial. 

Note that in the above expansion we get polynomials in the variable 
$\DS k_{N_0,N}=\frac{k-\bbE[S_N-S_{N_0}]}{\sig_{N,N_0}}$, 
not in the variable $k_N=\frac{k-\bbE(S_N)}{\sig_N}$. 
Since $k_{N_0,N}=k_N\al_{N_0,N}+O(\ln \sig_N/\sig_N)$, where $\al_{N_0,N}=\sig_{N}/\sig_{N_0,N}=O(1)$, the binomial theorem shows that 
such polynomials can be rewritten as polynomials in the variable  
$k_N$ 
whose coefficients are uniformly bounded in $N$. 
We also remark that in the above expansions we get the 
exponential terms 
\[
e^{-\frac{(k-a_{N_0,N})^2}{2(V_N-V_{N_0})}}
\quad\text{where}\quad
a_{N_0,N}=\bbE[S_N-S_{N_0}]
\] 
and not $e^{-(k-a_N)^2/2V_N}$ (as claimed in Theorem \ref{IntIndThm}). 
In order to address this
 fix some $\ve<1/2$.   Note that for
 $|k-a_{N_0,N}|\geq  V_N^{\frac12+\ve}$ we have
 \[
e^{-\frac{(k-a_{N_0,N})^2}{2(V_N-V_{N_0})}}=o(e^{-cV_N^{2\ve}})
 \quad \text{and} \quad 
 e^{-\frac{(k-a_{N_0,N})^2}{2V_N}}=o(e^{-cV_N^{2\ve}})
 \text{ for some }c>0.
  \]
 Since both terms are $o(\sig_N^{-s})$ for any $s$, 
 it is enough to explain how to replace 
 $\DS 
 e^{-\frac{(k-a_{N_0,N})^2}{2(V_N-V_{N_0})}}
 $
 with $\DS e^{-\frac{(k-a_{N})^2}{2V_N}}$ when 
 $|k-a_{N_0,N}|\leq  V_N^{\frac12+\ve}$ 
 (in which  case $\DS |k-a_N|=O(V_N^{\frac12+\ve})$). 
 For such $k$'s we can write
\begin{equation}\label{ExpTrans1}
\exp\left[-\frac{(k-a_{N_0,N})^2}{2(V_N-V_{N_0})}\right]=
\end{equation}
$$ \exp\left[-\frac{(k-a_{N_0,N})^2}{2V_N}\right]\;
\exp\left[-\frac{(k-a_{N_0,N})^2 V_{N_0}}{2V_N(V_N-V_{N_0})}\right].
$$
Since $\DS \frac{(k-a_{N_0,N})^2 V_{N_0}}{2V_N(V_N-V_{N_0})}
=O\left(V_N^{-(1-3\ve)}\right)$,
for any $d_1$ we have
\begin{equation}\label{ExTrans1.1}
\exp\left[-\frac{(k-a_{N_0,N})^2 V_{N_0}}{2V_N(V_N-V_{N_0})}\right]=
\end{equation}
$$\sum_{j=0}^{d_1}\frac{V_{N_0}^j}{2^j(V_N-V_{N_0})^j j!}
\left(\frac{(k-a_{N_0,N})^2}{\sig_N^2}\right)^j+O(V_N^{-(d_1+1)(2-3\ve)}).$$
Note that (using the binomial formula) the first term on the above right hand side
 is a polynomial of the variable $(k-a_{N})/\sig_N$ whose coefficients are uniformly bounded in $N$.

Next we analyze the first factor in the RHS of \eqref{ExpTrans1}.
As before, it is enough to consider $k$'s such that 
 $|k-a_N|\leq  V_N^{\frac12+\ve}$  for a sufficiently small $\ve.$
We have
\begin{equation}\label{Centring}
\exp\left[-\frac{(k-a_{N,N_0})^2}{2V_N}\right]=
 \end{equation}
$$ \exp\left[-\frac{(k-a_N)^2}{2V_N}\right] \exp\left[-\frac{2(k-a_N)a_{N_0}+a_{N_0}^2}{2V_N}\right]. $$
 Note that
$\frac{(k-a_N)a_{N_0}+a_{N_0}^2}{2V_N}
=k_N\beta_{N_0,N}+\theta_{N_0,N}$, where 
$$\beta_{N_0,N}=\frac{a_{N_0}}{2\sig_N}=O\left(\frac{\ln \sig_N}{\sig_N}\right)
\;\text{ and }\;
\theta_{N_0,N}=\frac{a_{N_0}^2}{2V_N}=O\left(\frac{\ln^2\sig_N}{V_N}\right).$$


Approximating $e^{\frac{(k-a_N)a_{N_0}+a_{N_0}^2}{2V_N}}$ by a polynomial of a sufficiently large degree $d_2$ in the variable $\frac{(k-a_N)a_{N_0}+a_{N_0}^2}{2V_N}$
completes the proof of existence of polynomials $P_{a,b,N}$  claimed in the theorem
(the Taylor reminder in the last approximation is of order 
$\DS O\left(V_N^{-d_2(\frac12-\ve)}\right)$, 
so we can take $d_2=4(r+1)$ assuming that $\ve$ is small enough). 
 
Finally, let us show that the coefficients of the polynomials $P_{a,b,N}$ 
constructed above are uniformly bounded in $N$. In fact, we will show that for each nonzero resonant point $t_j=2\pi l/m$,
the coefficients of the polynomials coming from integration over $I_j$ are of order 
$$O\left((1+M_N^{q_0}(m))e^{-M_N(m)}\right),$$ where $q_0=q_0(r)$ depends only on $r$. 

Observe that the additional contribution  to the coefficients  of the polynomials coming from the transition between the variables $k_N$ and $k_{N_0,N}$ is uniformly bounded in $N$.
Hence we only need to show that the coefficients of the (original) polynomials in the variable $k_{N_0,N}$ are uniformly bounded in $N$. The possible largeness of these coefficient  can only come from the terms $C_{u,N,t_j}$, for $u=0,1,2,...,d$ which are of order $M_N^u(m)$, respectively. However, the corresponding terms are multiplied by terms of the form $\Phi_{N_0,N}(t_j)\Phi_{N_0}^{(\ell)}(t_j)$ for certain $\ell$'s which are uniformly bounded in $N$ (see also \eqref{Cj(k)}).
We conclude 
that there are constants $W_j\in\bbN$ and $a_j\in\bbN$ which depend only on $t_j$ and $r$ so that the coefficients of the resulting polynomials are composed of a sum of at most $W_j$ terms of order $(M_N(m))^{a_j}\Phi_{N_0,N}(t_j)\Phi_{N_0}^{(\ell)}(t_j)$, where $\ell\leq E(r)$ for some $E(r)\in\bbN$. 
Next, we have   
\begin{equation}
\label{DerComb}
\Phi_{N_0}^{(\ell)}(t_j)\Phi_{N_0, N}(t_j)=
\end{equation}
$$\sum_{n_1,\dots, n_k\leq N_0; \atop \ell_1+\dots+\ell_k=\ell}
\gamma_{\ell_1,\dots, \ell_k}
\left(\prod_{q=1}^k 
\phi_{n_q}^{(\ell_q)}(t_j) \right)
\left[\prod_{n\leq N, \; n\neq n_k} \phi_n(t_j)\right]$$
where $\gamma_{\ell_1,\dots, \ell_k}$ are bounded coefficients of combinatorial nature.
 Using (\ref{Roz0}) we see that 
 for each $n_1,\dots, n_k$ the product in the square brackets
is at most $C e^{-c_0 M_N(m)+O(1)}$ for some $C, c_0>0$. Hence
$$|\Phi_{N_0}^{(\ell)}(t_j)\Phi_{N_0, N}(t_j)|\leq 
\hat{C} N_0^{\ell} \; e^{-c_0 M_N(m)},\,\,\hat C>0.$$
Now, observe that the definition of $N_0$ gives
$M_N(m)\geq \ve_0 N_0$, $\ve_0>0$. Therefore
$\DS |\Phi_{N_0}^{(\ell)}(t_j)\Phi_{N_0, N}(t_j)|\leq 
C_0 M_N^\ell (m)e^{-c_0 M_N(m)}$, and so each one of the above coefficients is of order  $M_N^{\ell'}(m)e^{-c_0 M_N(m)}$ for some $\ell'$ which does not depend on $N$.
\qed

\begin{remark}
The transition between the variables $k_{N_0,N}$ and $k_{N}$ changes the monomials of the polynomials $P_{a,b,N}$,  $a\not=0$  coming from integration over $I_j$, for $t_j\not=0$ into  monomials of the form $\frac{c_N a_{N_0}^{j_1}\sig_{N_0}^{j_2}k_N^{j_3}}{\sig_N^{u}}$ for some bounded sequence $(c_N)$, $j_1,j_2,j_3\geq0$  and $u\in\bbN$. As we have explained, the coefficients of these monomials are uniformly bounded. Still, it seems  more natural to consider such monomials as part of the polynomial $P_{a,b+u,N}$. In this case we still get polynomials with bounded coefficients since $a_{N_0}$ and $\sig_{N_0}$ are both $O(N_0)$, $N_0=O(M_{N}(m))$ and $c_N$ contains a term of the form $\Phi_{N_0}^{(\ell)}(t_j)\Phi_{N_0, N}(t_j)$. 
\end{remark}

\begin{remark}
As can be seen from the proof, the resulting expansions might contain terms corresponding to 
$\sig_N^{-s}$ for $s>r$. 
Such terms can be disregarded. For $\frac{|k-a_N|}{\sig_N}\leq V_N^\ve$ this follows
because the coefficients of our exapansions are $O(1)$ and
for $\frac{|k-a_N|}{\sig_N}\geq V_N^\ve$ this follows from \eqref{expo}. 
In practice, some of the polynomials $P_{a,b, N}$ with $b\leq r$ might
 have coefficients which are $\DS o(\sig_N^{b-r})$ (e.g. when $b+u>r$ in the last remark) so they also can be
 disregarded. The question when the terms $P_{a,b, N}$ may be disregarded
  is in the heart of the proof of 
 Theorem \ref{r Char} given in the next section.
 \end{remark}

\subsection{A summary}
 The proofs of Proposition \ref{PrLLT-SLLT}, Theorem \ref{ThLLT} and Theorem \ref{r Char} will be based on careful analysis of the formulas of the polynomials from Theorem \ref{IntIndThm}. For this purpose, it will be helpful
to summarize the main conclusions
 from the proof of Theorem \ref{IntIndThm}.
Let $r\geq1$ and $t_j=2\pi l/m$ be a nonzero resonant point.  Then the arguments in the proof of Theorem \ref{IntIndThm} yield that the contribution to the expansion coming from $t_j$ is 
\begin{equation}\label{Cj(k)}
\textbf{C}_j(k):=
\end{equation}
$$ e^{-it_j k}\Phi_{N_0,N}(t_j)\sum_{s\leq r-1}\left(\sum_{u+l=s}\frac{\Phi_{N_0}^{(l)}(t_j)C_{u,N}}{l!}\right)\int_{U_j}e^{-ihk}h^{s}\Phi_{N_0,N}(h)dh $$
where $U_j=I_j-t_j$, $C_{u, N}$ are given by \eqref{EqOrderD} and 
$C_{0,N}=1$.  When $t_j=0$ then it is sufficient to consider only $s=0$, $N_0=0$
and the contribution is just the integral
\[
\int_{-\del}^\del e^{-ihk}\Phi_{N}(h)dh
\]
where $\del$ is small enough.
As in (\ref{MainInt}), changing variables we can replace the integral corresponding to $h^s$ with 
\begin{eqnarray}\label{s}
\sig_{N_0,N}^{-s-1}\int_{-\infty}^\infty e^{-ih (k-\bbE[S_{N_0,N}])/\sig_{N_0,N}}h^s e^{-h^2/2}\\\times\left(1+\sum_{w=1}^r \frac{A_{w,N_0,N}(h)}{\sig_{N_0,N}^w}+\frac{h^{r+1}}{\sig_{N_0,N}^{r+1}} O(1)\right)dh.\nonumber
\end{eqnarray}
After that was established, the proof was completed using \eqref{Fourir} and some estimates whose whose purpose was to make the transition between the variables $k_{N_0,N}$ and $k_N$.

\section{Uniqueness of trigonometric expansions.}\label{Sec5}
In several proofs we will need the following result.
\begin{lemma}\label{Lemma}
Let $r\geq1$ and $d\geq0$. Set $\cR_0=\cR\cup\{0\}$
where $\cR$ is the set of nonzero resonant points.
For any $t_j\in\cR_0$, let $A_{0,N}(t_j)$,...,$A_{d,N}(t_j)$ be sequences 
so that, uniformly in $k$ such that $\DS k_N=\frac{k-\bbE(S_N)}{\sig_N}=O(1)$ we have
\[
\sum_{t_j\in\cR_0}e^{-it_j k}\left(\sum_{m=0}^{d}k_N^m A_{m,N}(t_j)\right)
=o(\sig_N^{-r}).
\]
Then for all $m$ and $t_j$
\begin{equation}\label{A def}
A_{m,N}(t_j)=o(\sig_N^{-r}).
\end{equation}
In particular the polynomials from the definition of the (generalized) Edgeworth expansions are unique up to terms of order $o(\sig_N^{-r})$.
\end{lemma}

\begin{proof}
The proof is by induction on $d$. Let us first set  $d=0$. Then, for any $k\in\bbN$ we have
\begin{equation}\label{d=0}
\sum_{t_j\in\cR_0}e^{-it_j k}A_{0,N}(t_j)=o(\sig_N^{-r}).
\end{equation}
Let $T$ be the number of nonzero resonant points, and let us relabel them as $\{x_1,...,x_T\}$. Consider the vector $$\fA_N=(A_{0,N}(0), A_{0,N}(x_1),...,A_{0,N}(x_T)).$$ Let
$\cV$ be the transpose of the Vandermonde matrix of the distinct numbers $\al_j=e^{-ix_j}, j=0,1,2,...,T$ where $x_0:=0$. Then $\cV$ is invertible and by considering $k=0,1,2,...,T$ in (\ref{d=0}) we see that (\ref{d=0}) holds true if and only if 
\[
\fA_N=\cV^{-1}o(\sig_N^{-r})=o(\sig_N^{-r}).
\]
Alternatively, let $Q$ be the least common multiple of the denominators of $t_j\in\cR$.
Let $a_N(p)=A_{0,N}(2\pi p/Q)$ if $2\pi p/Q$ is a resonant point and $0$ otherwise. Then for $m=0,1,...,Q-1$ we have
\[
\hat a_N(m):=\sum_{p=0}^{Q-1}a_N(p)e^{-2\pi pm/Q}=o(\sig_N^{-r}).
\]
Therefore, by the inversion formula of the discrete Fourier transform,
\[
a_N(p)=Q^{-1}\sum_{m=0}^{Q-1} \hat a_N(m) e^{2\pi i m p/Q}=o(\sig_N^{-r}).
\]

Assume now that 
the theorem is true for some $d\geq 0$ and any sequences functions $A_{0,N}(t_j),...,A_{d,N}(t_j)$. Let $A_{0,N}(t_j),...,A_{d+1,N}(t_j)$ be sequences so that  uniformly in $k$ 
such that $\DS k_N:=\frac{k-\bbE(S_N)}{\sig_N}=O(1)$ we have
\begin{equation}\label{I}
\sum_{t_j\in\cR_0}e^{-it_j k}\left(\sum_{m=0}^{d+1}k_N^m A_{m,N}(t_j)\right)=o(\sig_N^{-r}).
\end{equation}
Let us replace $k$ with $k'=k+[\sig_N]Q$, where $Q$ is the least common multiply of all the denominators of the
nonzero $t_j$'s. Then  $e^{-it_jk}=e^{-it_j k'}$. Thus,
\[
\sum_{t_j\in\cR_0}e^{-it_j k}\left(\sum_{m=0}^{d+1}(k_N'^m-k_N^{m})A_{m,N}(t_j)\right)=o(\sig_N^{-r}).
\]
Set $L_N=[\sig_N]Q/\sig_N\thickapprox Q$.
Then the  LHS above equals
\[
L_N\sum_{t_j\in\cR_0}e^{-it_j k}\left(\sum_{s=0}^{d}k_N^s\cA_{s,N}(t_j)\right)
\]
where
\[
\cA_{s,N}(t_j)=\sum_{m=s+1}^{d+1}A_{m,N}(t_j)L_N^{m-s-1}.
\]
By the induction hypothesis we get that 
\[
\cA_{s,N}(t_j)=o(\sig_N^{-r})
\]
for any $s=0,1,...,d$. In particular 
\[
\cA_{d,N}(t_j)=A_{d+1,N}(t_j)=o(\sig_N^{-r}).
\] 
Substituting this into \eqref{I} we can disregard 
the last term $A_{d+1,N}(t_j)$. Using the induction hypothesis with $A_{0,N}(t_j),A_{1,N}(t_j),...,A_{d,N}(t_j)$ we 
 obtain \eqref{A def}.
\end{proof}

\section{First order expansions}\label{FirstOrder}
In this section we will consider the case $r=1$. 
By \eqref{Cj(k)} and \eqref{s}, we see that the contribution coming from the integral over $I_j$ is
\[
\sig_{N_0,N}^{-1}e^{-it_j k}\Phi_{N}(t_j)\sqrt {2\pi} e^{-k_{N_0,N}^2/2}+o(\sig_N^{-1})
\]
where  $k_{N_0,N}=(k-\bbE(S_{N_0,N}))/\sig_{N_0,N}$. Now, using the arguments at the end of the proof of Theorem \ref{IntIndThm} when $r=1$ we can just replace $e^{-k_{N_0,N}^2/2}$ with $e^{-(k-\bbE(S_N))^2/2V_N}$ (since it is enough to consider the case when $k_{N_0,N}$ and $k_{0,N}$ are of order $V_N^\ve$). Therefore, taking into account that $\sig_{N_0,N}^{-1}-\sig_{N}^{-1}=O(\sig_N^{-2}N_0)$
we get that 
\begin{equation}\label{r=1}
\sqrt 2\pi\bbP(S_N=k)=
\end{equation}
$$\left(1+\sum_{t_j\in \cR}e^{-it_j k}\Phi_{N}(t_j)\right)\sig_{N}^{-1}e^{-(k-\bbE[S_N])^2/2V_N}+o(\sigma_N^{-1}). $$

Here $\cR$ is the set of all nonzero resonant points $t_j=2\pi l_j/m_j$.
 Indeed \sout{for} the contribution of the resonant points satisfying 
$M_N(m_j)\leq  R(r,K)\ln V_N$ is analyzed in \S \ref{Fin}.
 The contribution 
of the other nonzero resonant points $t$ is $o(\sig_N^{-1})$ due to \eqref{Roz0} in Section \ref{ScEdgeLogProkh}. In particular, \eqref{Roz0}
implies that $\Phi_N(t)=o(\sig_N^{-1})$ so adding the points with $M_N(m_j)\geq  R(r,K)\ln V_N$
only changes the sum in the RHS of \eqref{r=1} by 
$o(\sigma_N^{-1}). $

\begin{corollary}\label{FirstOrCor}
The local limit theorem holds if and only if $\DS \max_{t\in R}|\Phi_N(t)|=o(1)$.
\end{corollary}

\begin{proof}
It follows from (\ref{r=1}) that the LLT holds true if and only if for any $k$ we have
$$
\sum_{t_j\in \cR}e^{-it_j k}\Phi_{N}(t_j)=o(1).
$$
Now, the corollary follows from Lemma \ref{Lemma}. 
\end{proof}

Before proving Theorem \ref{ThLLT} we recall a standard fact which will also be useful 
in the proofs of Theorems \ref{Thm SLLT VS Ege} and \ref{Thm Stable Cond}.

\begin{lemma}
\label{LmUnifFourier}
Let $\{\mu_N\}$ be a sequence of measures probability measures on $\bbZ/m\bbZ$ and $\{\gamma_N\}$ be a 
positive sequence.
Then $\mu_N(a)=\frac{1}{m}+O(\gamma_N)$ for all $a\in \bbZ/m\bbZ$
if and only iff  $\hat\mu_N(b)=O(\gamma_N)$ for all $b\in \left(\bbZ/m\bbZ\right)\setminus \{0\}$
where $\hat\mu$ is the Fourier transform of $\mu.$ 
\end{lemma}

\begin{proof}
If $\mu_N(a)=\frac{1}{m}+O(\gamma_n)$ then
$$ \hat\mu_N(b)=\sum_{a=0}^{m-1} \mu_N(a) e^{2\pi i ab/m}=
\sum_{a=0}^{m-1} \frac{1}{m} e^{2\pi iab/m}+O(\gamma_N)=O(\gamma_N).$$
Next $\hat\mu_N(0)=1$ since $\mu_N$ are probabilities. Hence
if $\hat\mu_N(b)=O(\gamma_N)$ for all 
$b\in \left(\bbZ/m\bbZ\right)\setminus \{0\}$ then
$$ \mu_N(a)=\frac{1}{m} \sum_{b=0}^{m-1} \hat\mu_N(b) e^{-2\pi i ba/m}=
\frac{1}{m} \left[1+\sum_{b=1}^{m-1} \hat\mu_N(b) e^{-2\pi i  ba/m}\right]=\frac{1}{m}+O(\gamma_N)$$
as claimed.
\end{proof}

\begin{proof}[Proof of Theorem \ref{ThLLT}]
The equivalence of conditions (b) and (c) comes from 
the fact that for non-resonant points
the characteristic  function decays faster than any power of $\sigma_N$ (see \eqref{NonResDec}).

The equivalence of (a) and (c) is due to Corollary \ref{FirstOrCor}. 
Finally,  the equivalence between (c) and (d) comes from Lemma \ref{LmUnifFourier}.
\end{proof}

\begin{remark}
Theorem \ref{ThLLT} can also be deduced
from 
\cite[Corollary 1.4]{Do}. Indeed the corollary says that either the LLT holds or there is
an integer
$h\in (0, 2K)$ and a bounded sequence $\{a_N\}$ such that 
the limit 
$$ \bp(j)=\lim_{N\to\infty} \bbP(S_N-a_N=j \text{ mod } h) $$
exists and moreover if $k-a_n\equiv j \text{ mod }h$ then
$$ \sigma_N \bbP(S_N=k)=\bp(j) h \fg\left(\frac{k-\bbE(S_N)}{\sigma_N}\right)+o\left(\sigma_N^{-1}\right).$$
Thus in the second case the LLT holds iff $\bp(j)=\frac{1}{h}$ for all $j$ which is equivalent 
to $S_N$ being asymptotically uniformly distributed mod $h$ and also to 
the Fourier transform of $\bp(j)$ regarded as the measure on $\bbZ/(h\bbZ)$ being the 
$\delta$ measure at 0. Thus the conditions (a), (c) and (d) of the theorem are equivalent.
Also by the results of \cite[Section 2]{Do} 
(see also [\S 3.3.2]\cite{DS})
if $\bbE\left(e^{i\xi S_N}\right)$ does not converge to 
0 for some non zero $\xi$ then $\DS \left(\frac{2\pi}{\xi}  \right) \bbZ\bigcap 2\pi \bbZ$ is
a lattice in $\bbR$ which implies that $\xi$ is resonant, so condition (b) of the theorem is
also equivalent to the other conditions.
\end{remark}

\begin{proof}[Proof of Proposition \ref{PrLLT-SLLT}.]
Let $S_N$ satisfy LLT. Fix $m\in \mathbb{N}$ and suppose that 
$\DS \sum_n q_n(m)<\infty.$ Let $s_n$ be the most likely residue of $X_n$ mod $m$.
Then for $t=\frac{2\pi l}{m}$ we have
$$\phi_n(t)=e^{i t s_n}-\sum_{j\not\equiv s_n\; \text{mod}\; m} 
\bbP(X_n\equiv j\; \text{mod m})\left(e^{its_n}-e^{itj}\right),$$
so that $1\geq |\phi_n(t)|\geq 1-2 m q_n(m).$ It follows that for each $\ve>0$ there is 
$N(\ve)$ such that $\DS \left|\prod_{n=N(\ve)+1}^\infty \phi_n(t)\right|>1-\ve. $
Applying this for $\ve=\frac{1}{2}$ we have
\begin{equation}
\label{PhiNHalf}
\frac{1}{2}\leq\liminf_{N\to\infty} \left|\Phi_{N(1/2), N}(t)\right|\neq 0. 
\end{equation}
On the other hand the LLT implies that 
\begin{equation}
\label{PhiLim}
\lim_{N\to\infty} \Phi_N(t)=0.
\end{equation}
Since $\Phi_N=\Phi_{N(1/2)} \Phi_{N(1/2), N}$,
\eqref{PhiNHalf} and \eqref{PhiLim} imply that $\Phi_{N(1/2)}(t)=0.$
Since $\DS \Phi_{N(1/2)}\left(\frac{2\pi l}{m}\right)=\prod_{n=1}^{N(1/2)} \phi_n\left(\frac{2\pi l}{m}\right)$ 
we conclude that there exists $n_l\leq N(1/2)$ such that 
$\phi_{n_l}(\frac{2\pi l}{m})=0.$ 
Hence
$Y=X_{n_1}+X_{n_2}+\dots X_{n_{m-1}}$ satisfies
$\bbE\left(e^{2\pi i (k/m)Y}\right)=0$ for $k=1,\dots m-1.$
By Lemma \ref{LmUnifFourier} both $Y$ and $S_N$ for $N\geq N(1/2)$ are uniformly distributed.
This proves the proposition.
\end{proof}

\section{Characterizations of Edgeworth expansions of all orders.}
\label{ScCharacterization}

\subsection{Derivatives of the non-perturbative factor.}

Next we prove the following result.
\begin{proposition}\label{Thm}
Fix  $r\geq1,$ and assume that $M_N\leq R(r, K)\ln\sig_N$  (possibly along a subsequence).
Then
Edgeworth expansions of order $r$ hold true (i.e. (\ref{EdgeDef}) holds for such $N$'s) iff
for each $t_j\in\cR$ and $0\leq\ell<r$ (along the underlying subsequence) we have
\begin{equation}\label{Cond}
\sig_N^{r-\ell-1}\Phi_{N_0,N}(t_j)\Phi_{N_0}^{(\ell)}(t_j)=o(1).
\end{equation}
\end{proposition}

\begin{proof}
First, in view of \eqref{Cj(k)}  and \eqref{expo},  it is clear
 that the condition (\ref{Cond}) is sufficient 
for expansions of order $r$.

Let us now prove that the condition (\ref{Cond}) is necessary for the expansion of order $r.$ 

We will use induction on $r$. For $r=1$ (see \eqref{r=1}) our expansions read
\[
\bbP(S_N=k)=\sig_N^{-1}e^{-k_N^2/2}
\left[1+\sum_{t_j\in \cR}e^{-it_j k}\Phi_N(t_j)\right]+o(\sig_N^{-1}).
\]
Therefore if 
\[
\bbP(S_N=k)=\sig_N^{-1}e^{-k_N^2/2}P_N(k_N)+o(\sig_N^{-1})
\]
 for some  polynomial $P_N$ 
 then Lemma \ref{Lemma} tells us that,
 in particular $\Phi_{N}(t_j)=o(1)$ for each $t_j\in\cR$.
 
Let us assume now that the necessity part in Proposition \ref{Thm} holds 
for $r'=r-1$ and  prove that it holds for $r$. We will use the following lemma.

\begin{lemma}\label{LemInd}
Assume that for some $t_j\in\cR$,
\begin{equation}\label{Ind}
\sig_{N}^{r-2-l}\Phi_{N_0,N}(t_j)\Phi_{N_0}^{(l)}(t_j)=o(1),\, l=0,1,...,r-2.
\end{equation}
  Then, up to an $o(\sig_N^{-r})$ error term, the contribution of  $t_j$ to the generalized Edgeworh expansions of order $r$ is
\begin{equation}\label{Val0}
e^{-it_j k}e^{-k_{N}^2/2}\left(\frac{\Phi_{N_0}(t_j)}{\sig_N}+\sum_{q=2}^{r}\frac{\mathscr H_{N,q}(k_{N})}{\sig_N^{q}}\right)
\end{equation}
with
\begin{equation}
\label{DefCH}
\mathscr H_{N,q}(x)=\mathscr H_{N,q}(x;t_j)=
\end{equation}
$$ \cH_{N,q,1}(x)+\cH_{N,q,2}(x)+\cH_{N,q,3}(x)+\cH_{N,q,4}(x)
$$
where
$$ \cH_{N, q,1}(x)=
\frac{(i)^{q-1}H_{q-1}(x)\Phi_{N_0,N}(t_j)\Phi_{N_0}^{(q-1)}(t_j)}{(q-1)!},$$
$$ \cH_{ N,q, 2}(x)=
\frac{(i)^{q-1}H_{q-1}(x)\Phi_{N_0,N}(t_j)\Phi_{N_0}^{(q-2)}(t_j)C_{1,N,t_j}}{(q-2)!},
$$
$$\cH_{N,q,3}(x)= \frac{a_{N_0}(i)^{q-2}H'_{q-2}(x)\Phi_{N_0,N}(t_j)\Phi_{N_0}^{(q-2)}(t_j)}{(q-2)!},$$
$$
\cH_{N,q,4}(x)=-\frac{xa_{N_0}(i)^{q-2}H_{q-2}(x)\Phi_{N_0,N}(t_j)\Phi_{N_0}^{(q-2)}(t_j)}{(q-2)!},
$$
 and $H_q$ are Hermite polynomials. 
 
 Here $C_{1,N,t_j}$ is given by \eqref{C 1 N} when $M_N(m)\leq R(K, r)\ln\sig_N$, and $C_{1,N,t_j}=0$
 when $M_N(m)>R(K,r)\ln\sig_N.$   (Note that in either case
 $C_{1,N,t_j}=O(M_N(m))=O(\ln\sig_N)$).
 
As a consequence, when the Edgeworth expansions of order $r$ hold true and \eqref{Ind} holds, then uniformly in $k$ so that $k_N=O(1)$ we have
\begin{equation}\label{Val}
\frac{\Phi_{N}(t_j)}{\sig_N}+\sum_{q=2}^{r}\frac{\mathscr H_{N,q}(k_{N};t_j)}{\sig_N^{q}}=o(\sig_N^{-r}).
\end{equation}
\end{lemma}

 The proof of the lemma will be given in \S \ref{SSSummingUp} after we finish 
the proof of  Proposition \ref{Thm}.

By the induction hypothesis  the condition \eqref{Ind} holds true.
Let us prove now that
for $\ell=0,1,2,...,r-1$ and $t_j\in\cR$ we have
\[
\Phi_{N_0,N}(t_j)\Phi_{N_0}^{(\ell)}(t_j)=o(\sig_N^{-r+1+\ell}).
\]
Let us write
\[
\frac{\Phi_{N}(t_j)}{\sig_N}+\sum_{q=2}^{r}\frac{\mathscr H_{N,q}(k_{N})}{\sig_N^{q}}=\sum_{m=0}^{r-1}k_N^m A_{m,N}(t_j).
\]
Applying Lemmas \ref{Lemma} and  \ref{LemInd} we get that
\[
A_{m,N}(t_j)=o(\sig_N^{-r})
\]
for each $0\leq m\leq r-1$ and $t_j\in\cR$. 

Fix $t_j\in\cR$. Using Lemma \ref{LemInd} and the fact that the Hermite polynomials $H_{u}$ have the same parity as $u$ and that their leading coefficient is $1$ we have
\begin{eqnarray}
\label{AEq1}
A_{r-1,N}(t_j)=\sig_N^{-r}(i)^{r-2}\Bigg(i\Phi_{N_0,N}(t_j)\Phi_{N_0}^{(r-1)}(t_j)/(r-1)\\+\Phi_{N_0,N}(t_j)\Phi_{N_0}^{(r-2)}(t_j)(iC_{1,N}-a_{N_0})\Bigg)=o(\sig_N^{-r}) \notag
\end{eqnarray}
and 
\begin{eqnarray}
\label{AEq2} 
\quad
A_{r-2,N}(t_j)=\sig_N^{-r+1}(i)^{r-3}\Bigg(i\Phi_{N_0,N}(t_j)\Phi_{N_0}^{(r-2)}(t_j)/(r-2)\\+\Phi_{N_0,N}(t_j)\Phi_{N_0}^{(r-3)}(t_j)(iC_{1,N}-a_{N_0})\Bigg)=o(\sig_N^{-r}).
\notag
\end{eqnarray}
Since $\Phi_{N_0,N}(t_j)\Phi_{N_0}^{(r-3)}(t_j)=o(\sig_N^{-1})$, 
\eqref{AEq2} yields
\[
\Phi_{N_0,N}(t_j)\Phi_{N_0}^{(r-2)}(t_j)=o(\sig_N^{-1}\ln\sig_N).
\]
Plugging this  into \eqref{AEq1} we get  
\[
\Phi_{N_0,N}(t_j)\Phi_{N_0}^{(r-1)}(t_j)=o(1).
\]
Therefore we can just disregard $\mathscr H_{r,N}(k_N;t_j)$ since 
 its coefficients are of order $o(\sig_N^{-r})$. 
Since the term $\cH_{r,N}(k_N;t_j)$ no longer appears, repeating the above arguments with $r-1$ in place of $r$ we have
\begin{eqnarray*}
A_{r-3,N}(t_j)=\sig_N^{-r+2}(i)^{r-4}\Bigg(i\Phi_{N_0,N}(t_j)\Phi_{N_0}^{(r-3)}(t_j)/(r-3)\\+\Phi_{N_0,N}(t_j)\Phi_{N_0}^{(r-4)}(t_j)(iC_{1,N}-a_{N_0})\Bigg)=o(\sig_N^{-r}).
\end{eqnarray*}
Since $\Phi_{N_0,N}(t_j)\Phi_{N_0}^{(r-4)}(t_j)=o(\sig_N^{-2})$, the above asymptotic equality yields that 
\[
\Phi_{N_0,N}(t_j)\Phi_{N_0}^{(r-3)}(t_j)=o(\sig_N^{-2}\ln\sig_N).
\]
Plugging this  into \eqref{AEq2} 
we get 
\[
\Phi_{N_0,N}(t_j)\Phi_{N_0}^{(r-2)}(t_j)=o(\sig_N^{-1}).
\]
Hence, we can disregard also the term $\mathscr H_{r-1,N}(k_N;t_j)$. 
Proceeding this way we get that 
$\DS 
\Phi_{N_0,N}(t_j)\Phi_{N_0}^{(\ell)}(t_j)=o(\sig_N^{\ell+1-r})
$
for any $0\leq \ell<r$.
\qed

Before proving Lemma \ref{LemInd}, let us state the following result, which is a consequence of Proposition \ref{Thm} and  \eqref{Roz0}.
\begin{corollary}\label{CorNoDer}
Suppose that for each nonzero resonant point $t$ we have  $\DS \inf_{n}|\phi_n(t)|>0$. Then for any $r$, the sequence $S_N$ obeys Edgeworth expansions of order $r$ if and only if $\Phi_N(t)=o(\sig_N^{1-r})$ for each nonzero resonant point $t$. 
\end{corollary}

\subsection{Proof of Lemma \ref{LemInd}.}
\label{SSSummingUp}
\begin{proof}
First, because of \eqref{Ind}, 
 for each $0\leq s\leq r-1,$ the terms indexed by $l<s-1$ in \eqref{Cj(k)}, 
  are of order $o(\sig_N^{-r})$
 and so they can be disregarded. Therefore, we need only to consider the terms indexed by $l=s$ and $l=s-1$. 
For such $l$, using again \eqref{Ind} we can disregard all the terms in \eqref{s} indexed 
by $w\geq1$, since the resulting terms are of order $o(\sig_N^{-w-r}\ln\sig_N)=o(\sig_N^{-r})$. 
 Now, since $\sig_{N_0,N}^{-1}-\sig_{N}^{-1}=O(V_{N_0}/\sig_N^3)$ we can 
 replace $\sig_{N_0,N}$ with $\sig_{N}^{-1}$ 
 in \eqref{Cj(k)}, as  the remaining terms are of order $o(\sig_N^{-r-1})$.
Therefore, using \eqref{Fourir} we get the following contributions from $t_j\in\cR$,
\begin{equation*}
e^{-it_j k}\;
e^{-k_{N_0,N}^2/2}\left(\frac{\Phi_{N}(t_j)}{\sig_N}+\sum_{q=2}^{r}\frac{\cH_{N,q}(k_{N_0,N})}{\sig_N^{q}}\right)
\end{equation*} 
 where $\cH_{N,q}(x)=\cH_{N,q,1}(x)+\cH_{N,q,2}(x)$  and $\cH_{N,q,j}, j=1,2$ 
 are defined after \eqref{DefCH}.
Note that when $x=O(1)$ and $q<r$,
\begin{equation}\label{Order}
\frac{\cH_{N,q,1}(x)}{\sig_N^q}=o(\sig_N^{-r+1})\,\,\text { and }\,\,\frac{\cH_{N,q,2}(x)}{\sig_N^q}=o(\sig_N^{-r}\ln\sig_N).
\end{equation}
while when $q=r$,
\begin{equation}\label{Order.1}
\frac{\cH_{N,r,1}(x)}{\sig_N^r}=O(\sig_N^{-r}\ln\sig_N)\,\,\text { and }\,\,\frac{\cH_{N,r,2}(x)}{\sig_N^r}=o(\sig_N^{-r}\ln\sig_N).
\end{equation}
Next
\[
k_{N_0,N}=(1+\rho_{N_0,N})k_N+\frac{a_{N_0}}{\sig_N}+\theta_{N_0,N}
\]
where $\rho_{N_0,N}=\sig_{N}/\sig_{N_0,N}-1=O(\ln\sig_N/\sig^2_N)$ and $$\theta_{N_0,N}=a_{N_0}\left(\frac{1}{\sig_{N_0,N}}-\frac{1}{\sig_N}\right)=O(\ln^2\sig_N/\sig_N^3).$$
Hence, 
when $|k_{N_0,N}|\leq\sig_N^{\ve}$ (and so 
$\DS k_N=O(\sig_N^{\ve})$) 
for some $\ve>0$ small enough then for each $m\geq 1$ 
we have 
\[
k_{N_0,N}^m=k_{N}^m+mk_{N}^{m-1}a_{N_0}/\sig_N+o(\sig_N^{-1}).
\]
Therefore, \eqref{Order} and \eqref{Order.1} show that 
 upon replacing $H_{q-1}(k_{N_0,N})$ with $H_{q-1}(k_{N})$ 
 the only additional term is
\[
\frac{a_{N_0}(i)^{q-1}H'_{q-1}(k_N)\Phi_{N_0,N}(t_j)\Phi_{N_0}^{(q-1)}(t_j)}{(q-1)!\sig_N^{q+1}}
\]
for $q=2,3,...,r-1$. We thus get that the  contribution of $t_j$ is
\begin{equation*}
e^{-it_j k} \;e^{-k_{N_0,N}^2/2}\left(\frac{\Phi_{N}(t_j)}{\sig_N}+\sum_{q=2}^{r}\frac{\cC_{N,q}(k_{N})}{\sig_N^{q}}\right)
\end{equation*} 
where 
\begin{equation*}
\cC_{N,q}(x)=\cH_{N,q}(x)+ \frac{a_{N_0}(i)^{q-2}H'_{q-2}(x)\Phi_{N_0,N}(t_j)
\Phi_{N_0}^{(q-2)}(t_j)}{(q-2)!}.
\end{equation*}
Note that $\cC_{N,2}(\cdot)=\cH_{N,2}(\cdot)$.
Finally, we can replace $e^{-k_{N_0,N}^2/2}$ with 
\[
(1-k_Na_{N_0}/\sig_N)e^{-k_{N}^2/2}
\]
since all other terms in the transition between  $e^{-k_{N_0,N}^2/2}$  to $e^{-k_{N}^2/2}$ are of order $o(\sig_N^{-1})$ (see (\ref{ExpTrans1}) and (\ref{ExTrans1.1})).
The term $-k_Na_{N_0}/\sig_N$ shifts the $u$-th order term to the $u+1$-th term, $u=1,2,...,r-1$ multiplied by $-k_Na_{N_0}$. 
 Next, relying on \eqref{Order} and \eqref{Order.1} we see that after multiplied by $k_Na_{N_0}/\sig_N$, the second term $\cH_{N,q,2}(k_N)$ from the definition of $\cH_{N,q}(k_N)$  is of order $o(\sig_N^{-r-1}\ln^2\sig_N)\sig_N^q$ and so 
 this product can be disregarded.
Similarly, we can ignore the additional contribution coming from multiplying the second term from the definition of $\cC_{N,q}(k_N)$ by $-k_Na_{N_0}/\sig_N$ (since this term is of order $o(\sig_N^{-r}\ln\sig_N)\sig_N^q$). 
We conclude that, up to a term of order $o(\sig_N^{-r})$, the total contribution of $t_j$ is 
\begin{equation*}
e^{-it_j k} \; e^{-k_{N}^2/2}\left(\frac{\Phi_{N}(t_j)}{\sig_N}+\sum_{q=2}^{r}\frac{\mathscr H_{N,q}(k_{N}; t_j)}{\sig_N^{q}}\right)
\end{equation*}
where 
$\DS 
\mathscr H_{N,q}(x;t_j)=\cC_{N,q}(x)-\frac{x a_{N_0}(i)^{q-2}H_{q-2}(x) \Phi_{N_0,N}(t_j)\Phi_{N_0}^{(q-2)}(t_j)}{(q-2)!}
$
which completes the proof of \eqref{Val0}.

Next we prove \eqref{Val}.
 On the one hand, by assumption we have  Edgeworth expansions or order $r$, and, on the other hand, 
we have the expansions from Theorem \ref{IntIndThm}.
Therefore, the difference between the two must be $o(\sig_N^{-r})$. Since the usual Edgeworth expansions contain no terms corresponding to nonzero resonant points, applying Lemma \ref{Lemma} and \eqref{Val0} we obtain \eqref{Val}.
\end{proof}

Note that the formulas of Lemma \ref{LemInd} together with already proven 
Proposition \ref{Thm}
give the following result.

\begin{corollary}
\label{CrFirstNonEdge}
Suppose that $\mathbb{E}(S_N)$ is bounded,   $S_N$ admits the Edgeworth expansion of order $r-1$, and, 
either 

(a) for some $\breps\leq 1/(8K)$
we have $N_0=N_0(N,t_j,\breps)=0$ for each each nonzero resonant point $t_j$, 

or (b) $\DS \varphi:=\min_{t\in\cR}\inf_{n}|\phi_n(t)|>0$.

Then 
$$ \sqrt{2\pi} \mathbb{P}(S_N=k)$$
$$=e^{-k_N^2/2} \left[\cE_r(k_N)+
\sum_{t_j\in\cR}\left(\frac{\Phi_N(t_j)}{\sig_N}+\frac{ik_N C_{1,N,t_j}\Phi_N(t_j)}{\sig_N^2}\right)  e^{-i t_j k}\right]+o(\sigma_N^{-r})$$
where $\cE_r(\cdot)$ is the Edgeworth polynomial of order $r$ (i.e. the contribution of
$t=0$), and we recall that 
$$iC_{1,N,t_j}=-\sum_{n=1}^N\frac{\bbE(e^{it_j X_n}\bar X_n)}{\bbE(e^{it_j X_n})}.$$
\end{corollary}

\begin{proof}
  Part (a) holds since under the assumption that $N_0=0$ all terms $\cH_{N,q,j}$ in \eqref{DefCH}
except $\cH_{N, 2, 2}$ vanish. 
 Part (b) holds since in this case the argument
proceeds similarly to the proof of Theorem \ref{IntIndThm} 
if we set $N_0=0$ for any $t_j$ (since we only needed $N_0$ to obtain  a positive lower bound on $|\phi_n(t_j)|$ for $t_j\in\cR$ and $N_0<n\leq N$).
\end{proof}

\begin{remark}
 Observe that $\DS \sig_N^{-1}\gg |C_{1, N, {t_j}}|\sig_N^{-2},$
so if the conditions of the corollary are satisfied but $|\Phi_N(t_j)|\leq c \sig_N^{1-r}$
(possibly along a subsequence), then the leading correction to the Edgeworth
expansion comes from
$$ e^{-k_N^2/2} \sum_{t_j\in\cR}\left(\frac{\Phi_N(t_j)}{\sig_N} \right).$$
Thus Corollary \ref{CrFirstNonEdge} 
strengthens Corollary \ref{CorNoDer} by computing the leading correction
to the Edgeworth expansion when the expansion does not hold.
\end{remark} 


\subsection{Proof of Theorem \ref{r Char}}
We will use the following.
\begin{lemma}\label{Lem}
Let $t_j$ be a nonzero resonant point, $r>1$ and suppose that $M_N\leq R\ln\sig_N$, $R=R(r,K)$ and that $|\bbE(S_N)|=O(\ln\sig_N)$. Then (\ref{Cond}) holds for  all $0\leq \ell<r$ if and only if 
\begin{equation}\label{CondDer}
|\Phi_{N}^{(\ell)}(t_j)|=o\left(\sig_N^{-r+\ell+1}\right)
\end{equation}
for  all $0\leq \ell<r$.
\end{lemma}
\begin{proof}
Let us first assume that (\ref{Cond}) holds.
Recall that 
\begin{equation}
\label{PhiNProduct}
\Phi_N(t)=\Phi_{N_0}(t) \Phi_{N_0, N}(t)
\end{equation}
with 
\begin{equation}
\label{TripleProduct}
\Phi_{N_0, N}(t)=\Phi_{N_0, N} (t_j)\Phi_{N_0, N}(h)  \Psi_{N_0, N}(h)
\end{equation}
 where $t=t_j+h$ and
\begin{equation}
\label{TripleFactorization}
 \Psi_{N_0, N}(h)=\exp\left[O(M_N(m))\sum_{u=1}^\infty (O(1))^u h^u\right]. 
\end{equation}

 For $\ell=0$ the result reduces to \eqref{PhiNProduct}.
For larger $\ell$'s we have
\begin{equation}\label{Relation}
\Phi_{N}^{(\ell)}(t_j)=\Phi_{N_0,N}(t_j)\Phi_{N_0}^{(\ell)}(t_j)+\sum_{k=0}^{\ell-1}\binom{\ell}{k}\Phi_{N_0,N}^{(\ell-k)}(t_j)\Phi_{N_0}^{(k)}(t_j).
\end{equation}
Fix some $k<\ell$. Then by \eqref{TripleProduct},
\begin{eqnarray*}
\Phi_{N_0, N}^{(\ell-k)}(t_j)=\Phi_{N_0,N}(t_j)\sum_{u=0}^{\ell-k}\binom{\ell-k}{u}\Phi_{N_0,N}^{(u)}(0)\Psi_{N_0,N}^{(\ell-k-u)}(0)\\=O(\ln^{\ell-k}\sig_N)\Phi_{N_0,N}(t_j)
\end{eqnarray*}
where we have used that $\bar S_{N_0,N}=S_{N_0,N}-\bbE(S_{N_0,N})$ satisfies
$|\bbE[(\bar S_{N_0,N})^q]|\leq C_q\sig_{N_0,N}^{2q}$, 
(see \eqref{CenterMoments}).
Therefore
\begin{equation}\label{Thereofre}
\Phi_{N_0,N}^{(\ell-k)}(t_j)\Phi_{N_0}^{(k)}(t_j)=O(\ln^{\ell-k}\sig_N)\Phi_{N_0,N}(t_j)\Phi_{N_0}^{(k)}(t_j).
\end{equation}
Finally, by (\ref{Cond}) we have 
\[
\Phi_{N_0,N}(t_j)\Phi_{N_0}^{(k)}(t_j)=o(\sig_N^{k+1-r})
\]
and so, since $k<\ell$,
\[
\Phi_{N_0}^{(\ell-k)}(t_j)\Phi_{N_0}^{(k)}(t_j)=o(\sig_N^{\ell+1-r}).
\]
This completes the proof that \eqref{CondDer} holds.

Next, suppose that (\ref{CondDer}) holds  for each $0\leq\ell<r$. 
Let use prove by induction on $\ell$ that 
\begin{equation}\label{Above}
|\Phi_{N_0,N}\Phi_{N_0}^{(\ell)}(t_j)|=o\left(\sig_N^{-r+\ell+1}\right).
\end{equation}
For $\ell=1$ this follows  from \eqref{Relation}.
Now take $\ell>1$ and assume that \eqref{Above} holds with $k$ in place of $\ell$ for each $k<\ell$. Then by (\ref{Relation}), (\ref{Thereofre})  and the induction hypothesis we get that 
\[
\Phi_{N}^{(\ell)}(t_j)=\Phi_{N_0,N}(t_j)\Phi_{N_0}^{(\ell)}(t_j)+o(\sig_N^{\ell+1-r}).
\]
By assumption we have $\Phi_{N}^{(\ell)}(t_j)=o(\sig_N^{\ell+1-r})$ and hence
\[
\Phi_{N_0,N}(t_j)\Phi_{N_0}^{(\ell)}(t_j)=o(\sig_N^{\ell+1-r})
\]
as claimed.
\end{proof}

Theorem \ref{r Char}  in the case $M_N\leq R\ln \sigma_N$
follows now by first replacing $X_n$ with $X_n-c_n$, where $(c_n)$ is a bounded sequence of integers so that $\bbE[S_N-C_N]=O(1)$, 
where
\begin{equation}
\label{CNInteger}
 C_N=\sum_{j=1}^n c_j
\end{equation} 
  (see Lemma 3.4 in \cite{DS}), and then
applying Lemma~\ref{Lem} and Proposition~\ref{Thm}.

It remains to consider the case  when $M_N(m)\geq \brR \ln \sigma_N$ where $\brR$ is large enough. In that case, by Theorem \ref{ThEdgeMN}, the Edgeworth expansion of order $r$ hold true, and so, 
after the reduction to the case when $\bbE(S_N)$ is bounded, it is enough to show that
$\DS \Phi_N^{(\ell)}(t_j)=o\left(\sigma_N^{-r+\ell+1}\right)$ for all $0\leq \ell<r.$
By the arguments of Lemma \ref{Lem} it suffices to show that for each  $0\leq \ell<r$
we have
$\Phi_{N_0}^{(\ell)}(t_j)\Phi_{N_0, N}(t_j)
=o(\sigma_N^{-r})$.
To this end we write 
$$\Phi_{N_0}^{(\ell)}(t_j)\Phi_{N_0, N}(t_j)
=\sum_{n_1,\dots, n_k\leq N_0; \atop \ell_1+\dots+\ell_k=\ell}
\gamma_{\ell_1,\dots, \ell_k}
\left(\prod_{q=1}^k 
\phi_{n_q}^{(\ell_q)}(t_j) \right)
\left[\prod_{n\leq N, \; n\neq n_k} \phi_n(t_j)\right]$$
where $\gamma_{\ell_1,\dots, \ell_k}$ are bounded coefficients of combinatorial nature.
 Using (\ref{Roz0}) we see that 
 for each $n_1,\dots, n_k$ the product in the square brackets
is at most $C e^{-c_0 M_N(m)+O(1)}$ for some $C, c_0>0$. Hence
$$|\Phi_{N_0}^{(\ell)}(t_j)\Phi_{N_0, N}(t_j)|\leq 
\hat{C} N_0^{\ell} \; e^{-c M_N(m)}.$$
It remains to observe that the definition of $N_0$ gives
$M_N(m)\geq \hat\ve N_0.$ Therefore
$\DS |\Phi_{N_0}^{(\ell)}(t_j)\Phi_{N_0, N}(t_j)|\leq 
C^* M_N^\ell (m) \; e^{-c M_N(m)}=o(\sigma_N^{-r})$ 
provided that $M_N\geq \brR\ln \sigma_N$ for $\brR$ large enough. 
\qed

\section{Edgeworth expansions and uniform distribution.}

\subsection{Proof of Theorem \ref{Thm SLLT VS Ege}}
\label{SSEdgeR=2}
In view of Proposition \ref{Thm} with $r=2$, it is enough to show that if 
$\Phi_{N}(t_j)=o(\sigma_N^{-1})$ then the SLLT implies that 
\begin{equation}
\label{C11-C02}
|\Phi_{N_0,N}(t_j)\Phi_{N_0}'(t_j)|=o(1)
\end{equation}
 for any non-zero resonant point $t_j$ (note that the equivalence of conditions (b) and (c) 
 of the theorem follows from Lemma \ref{LmUnifFourier}).


Denote
$\DS \Phi_{N; k}(t)=\prod_{l\neq k, l\leq N} \phi_l(t)$.

Let us first assume that $\phi_k(t_j)\not=0$ for all $1\leq k\leq N$. Then 
$\phi_k'(t_j) \Phi_{N; k}(t_j)=\phi_k'(t_j)\Phi_N(t_j)/\phi_k(t_j)$. Let $\ve_N=\frac{\ln\sig_N}{\sig_N}$. If for  all 
$1\leq k\leq N_0$ we have  $|\phi_k(t_j)|\geq \ve_N$ then 
$$
\left|\Phi_{N_0,N}(t_j)\Phi_{N_0}'(t_j)\right|=\left|\sum_{k=1}^{N_0}\phi_k'(t_j) \Phi_{N; k}(t_j)\right|\leq|\Phi_N(t_j)|\sum_{k=1}^{N_0}|\phi'_k(t_j)/\phi_k(t_j)| $$
$$ \leq C\ve_N^{-1} N_0 |\Phi_N(t_j)|
\leq C'\sig_N|\Phi_N(t_j)|\to 0 \text{ as }N\to\infty
$$
where we have used that $N_0=O(\ln V_N)$.
 Next suppose there is at least one $1\leq k\leq N_0$ such that
$|\phi_k(t_j)|<\ve_N$. Let us pick some $k=k_N$ with the latter property. Then for any $k\not=k_N$, $1\leq k\leq N_0$ we have
\[
|\phi_k'(t_j) \Phi_{N; k}(t_j)|\leq C|\phi_{k_N}(t_j)|<C\ve_N.
\] 
Therefore,
\[
\left|\sum_{k\not=k_N,\,1\leq k\leq N_0}\phi_k'(t_j) \Phi_{N; k}(t_j)\right|\leq \frac{C'\ln^2 \sig_N}{\sig_N}=o(1).
\]
It follows that 
\begin{equation}
\label{SingleTerm}
 \Phi_{N_0,N}(t_j)\Phi_{N_0}'(t_j)=\Phi_{N; k_N}(t_j) \phi_{k_N}'(t_j)+o(1).
 \end{equation}

 Next, in the case when $\phi_{k_0}(t_j)=0$ for some $1\leq k_0\leq N_0$, then \eqref{SingleTerm} clearly holds true with $k_N=k_0$ since all the other terms vanish.

In summary, either \eqref{C11-C02}
 holds or we have \eqref{SingleTerm}. In the later case,
using (\ref{Roz0}) we  obtain
\begin{equation}
\label{PhiKBound}
\left|\mathbb{E}\left(e^{i t_j S_{N; k_N}}\right)\right|\leq e^{-c_2\sum_{s\not=k_N, 1\leq s\leq N}q_s(m)}=e^{-c_2M_{N}(m)-q_{k_N}(m)}
\end{equation}
where $S_{N; k}=S_N-X_k$, and $c_2>0$ depends only on $K$. Since the SLLT holds true, $M_{N}$ converges to $\infty$ as $N\to\infty$. Taking into account that
$0\leq q_{k_N}(m)\leq1$  we get that 
the left hand side of \eqref{PhiKBound} converges to $0$,
proving \eqref{C11-C02}.
\end{proof}

\subsection{Proof of Theorem \ref{Thm Stable Cond}}
We start with the proof of part (1). Assume that  the LLT holds true in a superstable way. Let   $X_1',X_2',...$ be a square integrable integer-valued independent sequence which differs from $X_1,X_2,...$ by a finite number of elements. Then there is $n_0\in\bbN$ so that $X_n=X'_n$ for any $n>n_0$. 
Set  $\DS S_N'=\sum_{n=1}^N X'_n$, $Y=S'_{n_0}$ and 
$Y_N=Y\bbI(|Y|<\sig_N^{1/2+\ve})$, where $\ve>0$ is a small constant. By the Markov inequality we have 
$$
\bbP(|Y|\geq \sig_N^{1/2+\ve})=\bbP(|Y|^2\geq \sig_N^{1+2\ve})\leq\|Y\|_{L^2}^2\sig_N^{-1-2\ve}=o(\sig_N^{-1}).
$$
Therefore,  for any $k\in\bbN$ and $N>n_0$ we have
\begin{eqnarray*}
\bbP(S'_N=k)=\bbP(S_{N;1,2,...,n_0}+Y_N=k)+o(\sig_N^{-1})\\=\bbE[\bbP(S_{N;1,2,...,n_0}=k-Y_N|X_1',...,X_{n_0}')]+o(\sig_N^{-1})\\=
\bbE[P_{N:n_1,...,n_0}(k-Y_N)]+o(\sig_N^{-1})
\end{eqnarray*}
where $P_{N:n_1,...,n_0}(s)=\bbP(S_{N;1,2,...,n_0}=s)$ for any $s\in\bbZ$.
Since the LLT holds true in a superstable way, we have, 
uniformly in $k$ and the realizations of $X_1',...,X_{n_0}'$ that
$$
P_{N:n_1,...,n_0}(k-Y_N)=\frac{e^{-(k-Y_N-\bbE(S_N))^2/(2V_N)}}{\sqrt{2\pi}\sig_N}+o(\sig_N^{-1}).
$$
Therefore, 
\begin{equation}\label{Above1}
\bbP(S'_N=k)=
\end{equation}
$$ \frac{e^{-(k-\bbE(S_N))^2/2V_N}}{\sqrt{2\pi}\sig_N}
\bbE\big(e^{-(k-\bbE(S_N))Y_N/V_N+Y_N^2/(2V_N)}\big)+o(\sig_N^{-1}). $$
Next, since $|Y_N|\leq \sig_N^{1/2+\ve}$ we have that $\|Y_N^2/(2V_N)\|_{L^\infty}\leq \sig_{N}^{2\ve-1}$, and so when $\ve<1/2$ we have $\|Y_N^2/2V_N\|_{L^\infty}=o(1)$. 
 Recall that $k_N=(k-\bbE(S_N))/\sig_N$. Suppose first that
 $|k_N|\geq \sig_N^{\ve}$ with  $\ve<1/4.$  

Since 
$$\big|(k-\bbE(S_N))Y_N/V_N\big|\leq |k_N|\sig_N^{\ve-\frac12},$$ 
we get that the RHS of \eqref{Above1} is $o(\sig_N^{-1})$ (uniformly in such $k$'s). 

On the other hand, if  $|k_N|<\sig_N^{\ve}$ then $$\bbE\big(e^{-(k-\bbE(S_N))Y_N/V_N+Y_N^2/2V_N}\big)=1+o(1)$$ (uniformly in that range of $k$'s). 

We conclude that, uniformly in $k$, we have 
$$ 
\bbP(S'_N=k)=\frac{e^{-(k-\bbE(S_N))^2/(2V_N)}}{\sqrt{2\pi}\sig_N}+o(\sig_N^{-1}).
$$
Lastly, since $\DS \sup_{N}|\bbE(S_N)-\bbE(S_N')|\!\!<\!\infty$ and 
$\DS \sup_{N}|\text{Var}(S_N)-\text{Var}(S_N')|\!<\!\!\infty,$ 
$$ 
\bbP(S'_N=k)=\frac{e^{-(k-\bbE(S'_N))^2/(2V'_N)}}{\sqrt{2\pi}\sig'_N}+o(1/\sig_N')
$$
where $V_N'=\text{Var}(S_N')$ and $\sig_N'=\sqrt{V_N'}$.

Conversely, if the SLLT holds then 
$M_N(h)\to \infty$ for each $h\geq2$.
Now if $t$ is a nonzero resonant point with denominator $h$ then \eqref{Roz0} gives
$$ |\Phi_{N; j_1^{N}, j_2^{N},\dots ,j_{s_N}^{N}}(t)|\leq C e^{-c M_N(h)+\bar C \brs},\,
\,C,\bar C>0 $$
for any choice of $j_1^N,...,j_{s_N}^N$ and $\bar s$ with $s_N\leq \bar s$.
Since the RHS tends to 0 as $N\to\infty$, $\{X_n\}\in EeSS(1)$
completing the proof of part (1).

For part (2) we only need  to show that (a) is equivalent to (b) as the equivalence
of (b) and (c) comes from Lemma \ref{LmUnifFourier}. By replacing again $X_n$ with $X_n-c_n$ it is enough to prove the equivalency in the case when $\bbE(S_N)=O(1)$.
The proof that (a) and (b) are equivalent 
consists of two parts.
The first part is the following statement whose proof  is a straightforward adaptation of the proof of Theorem 
\ref{r Char}
and is therefore omitted.

\begin{proposition}
$\{X_n\}\in SsEe(r)$ if and only if for each $\brs$, each sequence
$j_1^N, j_2^N, \dots ,j_{s_N}^N$ with $s_N\leq \brs$, each
$\ell<r$ and each $t\in\cR$ we have
\begin{equation}
\label{PhiDerFM}
\Phi^{(\ell)}_{N; j_1^N, j_2^N, \dots, j_{s_N}^N}(t)=
o(\sigma_N^{\ell+1-r}). 
\end{equation}
\end{proposition}

Note that the above proposition shows that the condition  
$\DS \Phi_{N; j_1^N, j_2^N, \dots ,j_{s_N}^N}(t)=
o(\sigma^{1-r}_{N})$ is necessary.

The second part of the argument is to show that if 
$$ \Phi_{N; j_1^N, j_2^N, \dots ,j_{s_N}^N}(t)=
o(\sigma^{1-r}_{N}) $$
holds for every finite modification of $S_N$ with $s_N\leq \brs+\ell$ (uniformly) then 
\eqref{PhiDerFM} holds for every modifications with $s_N\leq \brs$
so that the condition  
$\DS \Phi_{N; j_1^N, j_2^N, \dots ,j_{s_N}^N}(t)=
o(\sigma^{1-r}_{N})$ is also sufficient.

To this end we introduce some notation. 
Fix a nonzero resonant point $t=\frac{2\pi l}{m}.$
Let
 $\check\Phi_{N}$ be the characteristic function of the sum $\check S_{N}$ of all $X_n$'s  
 such that  $1\leq n\leq N$, $n\not\in\{j_1^{N},j_2^{N},...,j_{s_N}^N\}$ and $q_{n}(m)\geq\bar\epsilon$. 
Let $\check N$ be the number of terms in $\check S_N.$
Denote
 $\tilde S_{N}=S_{N;j_1^N,j_2^N,\dots,j_{s_N}^{N}}-\check S_{N}$
 and let $\tilde\Phi_N(t)$ be the characteristic function of $\tilde S_N.$
 Similarly to the proof of Theorem \ref{r Char} it suffices to show
 that for each $\ell<r$
 $$ \left| \check \Phi_N^{(\ell)} \tilde \Phi_N(t)\right|=o(\sigma_N^{1+\ell-r}) $$
and, moreover, we can assume  that $M_N(m)\leq \brR \ln \sigma_N$
so that $\check N=O(\ln \sigma_N).$ 
We have (cf. \eqref{DerComb}) ,
$$ \check \Phi_N^{(\ell)} \tilde \Phi_N(t)
=\sum_{\substack{n_1, \dots ,n_k; \;\\ \ell_1+\dots+\ell_k=\ell}}
\prod_{w=1}^k \gamma_{\ell_1,\dots, \ell_k}
\phi_{n_w}^{(\ell_w)}(t_j) \prod_{n\not\in\{n_1,n_2 \dots ,n_k, j_1^N,j_2^N \dots, j_{s_N}^N\}}
 \phi_n(t_j)$$
 where the summation is over all tuples $n_1,n_2, \dots ,n_k$ such that 
 $q_{n_w}(m)\geq\bar\epsilon$.
 Note that the absolute value of each term in the above sum is bounded by
 $C |\Phi_{N; n_1, \dots ,n_k, j_1\dots, j_{s_N}^N}(t_j)|=o(\sigma_N^{1-r}).$
 It follows that the whole sum is 
$$ o\left(\sigma_N^{1-r} \check N^\ell\right) =
o\left(\sigma_N^{1-r} \ln^\ell \sigma_N \right)  $$
completing the proof. \qed

\begin{remark}
Lemma \ref{LmUnifFourier} and 
Theorem \ref{r Char} show that the convergence to uniform distribution on any
factor $\mathbb{Z}/h \mathbb{Z}$ with the speed $o(\sig_N^{1-r})$ is necessary for 
Edgeworth expansion of order $r.$ This is quite intuitive. Indeed calling $\mathscr E_r$
 the Edgeworth function of order  $r$, (i.e. the contribution from zero), then
it is a standard result from
numerical integration (see, for instance, \cite[Lemma A.2]{DNP}) 
that for each $s\in \mathbb{N}$ and each $j\in \mathbb{Z}$
$$ \sum_{k\in \mathbb{Z}} h \mathscr E_r\left(\frac{j+hk}{\sqrt{\sig_N}}\right)
=\int_{-\infty}^\infty \mathscr E_r(x) dx+o\left(\sig_N^{-s}\right)=
1+o\left(\sig_N^{-s}\right) $$
where in the last inequality we have used that the non-constant Hermite polynomials have zero mean with respect to the standard normal law (since they are orthogonal to the constant functions).
However, using this result to show that
$$\DS \sum_{k\in \mathbb{Z}} \mathbb{P}(S_N=j+kh)=\frac{1}{h}+o\left(\sig_N^{1-r}\right) $$
requires a good control on  large values of $k.$ While it appears possible to
obtain such control using the large deviations theory it seems simpler to 
 estimate the convergence rate towards uniform distribution from our generalized Edgeworth expansion.
\end{remark}

\section{Second order expansions}\label{Sec2nd}
In this section we will compute the polynomials in the general expansions in the  case $r=2$.

First, let us introduce some notations which depend on a resonant point $t_j.$
Let $t_j=2\pi l_j/m_j$ be a nonzero resonant point such that 
$M_N(m_j)\leq R(2,K)\ln V_N$ where $R(2,K)$ is specified in Remark \ref{R choice}.
 Let
 $\check\Phi_{j,N}$ be the characteristic function of the sum $\check S_{j,N}$ of all $X_n$'s  
 such that  $1\leq n\leq N$ and $q_{n}(m_j)\geq\bar\epsilon=\frac1{8K}$. 
 Note that $\check S_{j,N}$ was previously denoted by $S_{N_0}$. 
Let $\tilde S_{N,j}=S_N-\check S_{N,j}$ and denote by $\tilde \Phi_{N,j}$ its characteristic function.
(In previous sections we denoted the same expression by 
$S_{N_0,N}$, but here we want to emphasize the dependence on $t_j$.)
 Let  $\gamma_{N,j}$  be the ratio between the third moment of $\tilde S_{N,j}-\bbE(\tilde S_{N,j})$ and  its variance. 
Recall that by \eqref{CenterMoments}
$|\gamma_{N,j}|\leq C$ for some $C$. 
Also,  let $C_{1,N,t_j}$ be given by (\ref{C 1 N}), with the indexes rearranged so that the $n$'s with $q_n(m)\geq\bar\ve$ are the first $N_0$  ones ($C_{1,N,t_j}$ is at most of order $M_N(m)=O(\ln V_{N})$). 
For the sake of convenience, when  either $t_j=0$ or $M_N(m_j)\geq R(2,K)\ln V_N$ we set $C_{1,N,t_j}=0$, $\tilde S_{N,j}=S_N$ and $\check S_{N,j}=0$. In this case $\tilde\Phi_{N,j}=\tilde\Phi_{N}$ and $\check\Phi_{N,j}\equiv 1$.  Also denote
$k_N=(k-\bbE(S_N))/\sig_N,$ $\bar S_N=S_N-\bbE(S_N)$, and
$\gamma_N=\bbE(\bar S_N^3)/V_N$, 
($\gamma_N$ is bounded).

\begin{proposition}\label{2 Prop}
Uniformly in $k$, we have
\begin{eqnarray}\label{r=2'}
\sqrt {2\pi}\bbP(S_N=k)=\Big(1+\sum_{t_j\in \cR}e^{-it_j k}\Phi_{N}(t_j)\Big)\sig_{N}^{-1}e^{-k_N^2/2}\\-\sig_{N}^{-2}e^{-k_N^2/2}\left(\gamma_N k_N^3/6+\sum_{t_j\in \cR}e^{-it_j k}\tilde\Phi_{N,j}(t_j)P_{N,j}(k_N)\right)\nonumber\\+o(\sigma_N^{-2})\nonumber
\end{eqnarray} 
where 
$$
P_{N,j}(x)=\big(\check\Phi_{N,j}(t_j)(iC_{1,N,t_j}-\bbE(\check S_{N,j}))+i\check\Phi_{N,j}'(t_j)\big)x+\check\Phi_{N,j}(t_j)\gamma_{N,j}x^3/6.
$$
\end{proposition}

\begin{proof} 
Let $t_j=\frac{2\pi l}{m}$ be a resonant point with $M_N(m)\leq R(2,K) \ln V_N.$ 
Recall that $\textbf{C}_j(k)$  are given by \eqref{Cj(k)}.
First,
in order to compute the term corresponding to $\sig_{N_0,N}^{-2}$ we need only to consider the case $s\leq1$ in (\ref{s}).
Using \eqref{A 1,n .1}  we end up with 
 the following contribution
  of the interval containing $t_j$,
$$
\sqrt{(2\pi)^{-1}}e^{-it_j k}\tilde \Phi_{N,j}(t_j)\sig_{N_0,N}^{-1}\Bigg(\int_{-\infty}^{\infty}e^{-ih(k-\bbE[\tilde S_{N,j}])/\sig_{N,j}}e^{-h^2/2}dh$$
$$+\sig_{N,j}^{-1}\int_{-\infty}^\infty e^{-ih(k-\bbE[\tilde S_{N,j}])/\sig_{N,j}}\left(\frac{ih^3}{6}\bbE\left[\big(\tilde S_{N,j}-\bbE(\tilde S_{N,j})\big)^3\right]
\sig_{N,j}^{-3}\right)dh$$
$$+\sig_{N,j}^{-1}(C_{1,N}\check\Phi_{N,j}(t_j)+\check\Phi_{N,j}'(t_j))\int_{-\infty}^\infty e^{-ih(k-\bbE(\tilde S_{N,j}))/\sig_{N,j}}he^{-h^2/2}dh\Bigg)$$
$$=e^{-it_j k}\tilde\Phi_{N,j}(t_j)\sqrt{(2\pi)^{-1}} e^{-k_{N,j}^2/2}\sig_{N,j}^{-1}\Big(\check\Phi_{N,j}(t_j)+i\big(C_{1,N,t_j}\check\Phi_{N,j}(t_j)+\check\Phi_{N,j}'(t_j)\big)$$
$$\times k_{N,j}\sig_{N,j}^{-1}+\check\Phi_{N,j}(t_j)(k_{N,j}^3-3k_{N,j})\gamma_{N,j}\sig_{N,j}^{-1}/6\Big)
$$
where $\sig_{N,j}=\sqrt{V(\tilde S_{N,j})}$, $k_{N,j}=(k-\bbE(\tilde S_{N,j}))/\sig_{N,j}$ and $\gamma_{N,j}= \frac{\bbE[(\tilde S_{N,j}-\bbE(\tilde S_{N,j}))^3]}{\sig_{N,j}^2}$ (which is uniformly bounded).

As before we shall only consider the case where $|k_N|\leq V_N^\ve$
with $\ve=0.01$ since otherwise both the LHS and the RHS \eqref{r=2'}
are $O(\sig_N^{-r})$ for all $r.$
Then, the last display can be rewritten as $I+I\!\!I$ where
\begin{equation}
I=\frac{e^{-i t_j k }}{\sqrt{2\pi} \sig_{N,j}}\; e^{-k_{N,j}^2/2}\;  \Phi_N(t_j);
\end{equation}
$$ 
I\!\!I
=\frac{e^{-i t_j k }}{\sqrt{2\pi} \sig^2_{N,j}}\; e^{-k_{N,j}^2/2}\times$$
$$
\left[\Phi_N(t_j) \left(i C_{1, N,t_j} k_{N,j}+\frac{\gamma_{N,j}}{6} 
\left(k_{N,j}^3-3 k_{N,j}\right)\right)+i\check\Phi_{N,j}'(t_j) \tilde\Phi_{N,j}(t_j) k_{N,j}
\right].
$$

In the region $|k_N|\leq V_N^\ve$ we have
$$ I=
\frac{e^{-i t_j k }}{\sqrt{2\pi} \sig_{N}}\; e^{-k_{N}^2/2}\;  \left[1-q_{N,j} k_N \right]\Phi_N(t_j)+
o\left(\sig_N^{-2}\right)$$
where 
$$q_{N,j}=\bbE(\check S_{N,j})/\sig_{N,j}=\bbE(\check S_{N,j})/\sig_N+O(\ln\sig_N/\sig_N^3)=
O(\ln\sig_N/\sig_N)
$$
while
$$ I\!\!I=\frac{e^{-i t_j k }}{\sqrt{2\pi} \sig^2_{N}}\; e^{-k_{N}^2/2}\times$$
$$
\left[\Phi_N(t_j) \left(i C_{1, N,t_j} k_{ N}+\frac{\gam_{N,j} 
\left(k_{N}^3-3 k_{N}\right)}{6} \right)+i\check\Phi_{N,j}'(t_j) \tilde\Phi_{N,j}(t_j) k_{N}
\right]$$
$$+o\left(\sig_N^{-2}\right).$$
This yields (\ref{r=2'}) with $\cR_N$ in place of $\cR$, where   $\cR_N$ is the set on nonzero resonant points $t_j=2\pi l/m$ such that $M_N(m)\leq R(2,K)\ln V_N$. Next,
\eqref{Roz0} shows that if $M_N(m)\geq R(2,k)\ln V_N$ then
\[
\sup_{t\in I_j}|\Phi_{N}(t)|\leq e^{-c_0M_N(m)}=o(\sig_N^{-2})
\]
and so the contribution of $I_j$ to the right hand side of \eqref{r=2'} is $o(\sig_N^{-2}).$  Finally,
  the contribution coming from $t_j=0$ is 
\[
e^{-k_N^2/2}\left(\sig_N^{-1}+\sig_N^{-2}\gamma_N^3 k_N^3/6\right)
\]
and the proof of the proposition is complete.
\end{proof}
\begin{remark}\label{Alter 2nd Order}
 Suppose that $M_N(m)\geq R(2,K)\ln V_N$ and let $N_0$ is the number of $n$'s between $1$ to $N$ so that $q_n(m)\geq \frac{1}{8K}$. Then using (\ref{Roz0}) we also have 
$$
|\check\Phi'_{N,j}(t_j)\tilde \Phi_{N,j}(t_j)|\leq\sum_{n\in\cB_{\breps}(m)}|\bbE[X_ne^{it_j X_n}]|\cdot|\Phi_{N;n}(t_j)|
$$
$$
\leq CN_0(N,t_j,\breps)e^{-c_0 M_{N}(m)}\leq C'M_{N}(m)e^{-c_0 M_{N}(m)},
$$
where $$\cB_{N,\breps}(m)=\{1\leq n\leq N: q_n(m)>\breps\}.$$
Since $M_N(m)\geq R(2,K)\ln V_N$, for any $0<c_1<c_0$, when $N$ is large enough we have $$M_{N}(m)e^{-c_0 M_{N}(m)}\leq C_1e^{-c_1M_N(m)}=o(\sig_N^{-2}).$$ 
Similarly, $|\bbE(\check S_{N,j})\Phi_N(t_j)|=o(\sig_N^{-2})$ and 
$$C_{1,N,t_j}\Phi_{N}(t_j)=O(M_N(m))\Phi_{N}(t_j)=o(\sig_N^{-2}).$$ Therefore we get (\ref{r=2'}) when $\tilde S_{N,j}$ and $\check S_{N,j}$ are defined in the same way as in the case $M_N(m)\leq R(2,K)\ln V_N$.
\end{remark}

 Under additional assumptions the order 2 expansion can be simplified.

\begin{corollary}
\label{CrR2-SLLT}
If $S_N$ satisfies SLLT then
$$
\sqrt {2\pi}\bbP(S_N=k)=\frac{e^{-k_N^2/2}}{\sig_N}
\left(1+\sum_{t_j\in \cR}e^{-it_j k}\Phi_{N}(t_j)
-\frac{\gamma_N k_N^3}{6\sig_N}\right)
+o(\sigma_N^{-2}).
$$
\end{corollary}

\begin{proof}
The estimates  of \S \ref{SSEdgeR=2} together with \eqref{Roz0} show that if $S_N$ satisfies the SLLT then for all $j$
$$(1+M_N(m))\tilde\Phi_{N,j} (t_j) \check \Phi_{N,j}(t_j)=o(1) \text{ and }
\tilde\Phi_{N,j} (t_j) \check \Phi'_{N,j}(t_j)=o(1).$$
Thus all terms in the second line of \eqref{r=2'} except the first one make 
a negligible contribution,
and so they could be omitted.
\end{proof}

Next, assume that $S_N$ satisfies the LLT but not SLLT. According to 
Proposition \ref{PrLLT-SLLT}, in this case there exists $m$ such that 
$M_N(m)$ is bounded and
for $k=1,\dots ,m-1$ there exists $n=n(k)$ such that
$\phi_{n}(k/m)=0$. Let $\cR_s$ denote the set of nonzero resonant 
points $t_j=\frac{2\pi k}{m}$ so that
$M_N(m)$ is bounded and $\phi_{\ell_j}(t_j)=0$ for unique $\ell_j.$ 

\begin{corollary}\label{CrR2Z}
Uniformly in $k$, we have
\begin{eqnarray*}
\sqrt {2\pi}\bbP(S_N=k)=\Big(1+\sum_{t_j\in \cR}e^{-it_j k}\Phi_{N}(t_j)\Big)\sig_{N}^{-1}e^{-k_N^2/2}\\-\sig_{N}^{-2}e^{-k_N^2/2}\left(\gamma_N k_N^3/6+
\sum_{t_j\in \cR_s}i e^{-it_j k}\Phi_{N; \ell_j}(t_j) \phi_{\ell_j}'(t_j) k_N
\right)+o(\sigma_N^{-2}).\nonumber
\end{eqnarray*} 
\end{corollary}

\begin{proof}
As in the proof of Corollary \ref{CrR2-SLLT} we see that the contribution of the terms with
$k/m$ with $M_N(m)\to\infty$ is negligible. Next, for terms in $\cR_s$ the only non-zero
term in the second line in \eqref{r=2'}
corresponds to $\Phi_{N; \ell_j}(t_j) \phi_{\ell_j}'(t_j)$ while for the resonant points
such that $\phi_{\ell}(t_j)=0$ for two different $\ell$s all terms vanish.
\end{proof}

\section{Examples.}
\label{ScExamples}
\begin{example}
\label{ExNonAr}
Suppose $X_n$ are iid integer valued with step $h>1.$ That is there is $s\in \bbZ$ such that
$\bbP(X_n\in s+h \bbZ)=1$ and $h$ is the smallest number with this property. 
In this case \cite[Theorem 4.5.4]{IL} (see also \cite[Theorem 5]{Ess})
shows that there are polynomials $P_b$ such that
\begin{equation}
\label{EssEdge}
\bbP(S_N=k)= 
\sum_{b=1}^r \frac{P_{b} ((k-\bbE[S_N])/\sigma_N)}{\sigma_N^b}
\fg\left(\frac{k-\bbE(S_N)}{\sigma_N}\right)
+o(\sigma_N^{-r})
\end{equation}
for all $k\in sN+h\bbZ.$
Then
$$ \sum_{a=0}^{h-1} \sum_{b=1}^r e^{2\pi i a(k-sN)/h} \frac{P_{b} ((k-\bbE[S_N])/\sigma_N)}{\sigma_N^b}
\fg((k-\bbE(S_N))/\sigma_N)$$
provides $\DS o(\sigma_N^{-r})$ approximation to $\bbP(S_N=k)$ which is valid for {\em all}
$k\in \bbZ.$ 

Next let $\brS_N=X_0+S_N$ where $X_0$  is bounded and arithmetic with step 1.
Then using the identity
\begin{equation}
\label{Convolve}
 \bbP(\brS_N=k)=\sum_{u\equiv k-s N\text{ mod } h} \bbP(X_0=u) \bbP(S_N=k-u), 
\end{equation}

\noindent
invoking \eqref{EssEdge} and 
expanding $\DS \fg\left(\frac{k-u-\bbE(S_N)}{\sigma_N}\right)$ in the Taylor series
about $\frac{k-{ \bbE(S_N)}}{\sigma_N}$ we conclude that there are polynomials
$P_{b,j}$ such that we have for $k\in j+h\bbZ$,
$$ \bbP(\brS_N=k)= 
\sum_{b=1}^r \frac{P_{b, j} ((k-\bbE[S_N])/\sigma_N)}{\sigma_N^b}
\fg\left(\frac{k-\bbE(S_N)}{\sigma_N}\right)
+o(\sigma_N^{-r}). $$
Again
$$ \sum_{a=0}^{h-1} \sum_{j=0}^{h-1} e^{2\pi i a (k-j)/h} 
\sum_{b=1}^r \frac{P_{b,j} ((k-\bbE[S_N])/\sigma_N)}{\sigma_N^b}
\fg\left(\frac{k-\bbE(S_N)}{\sigma_N}\right) $$
provides the oscillatory expansion valid for all integers.
\end{example}

\begin{example}
\label{ExUniform}

Our next example is a small variation of the previous one.
Fix a positive integer $m.$ Let $X'$ be a random variable such that 
$X'$ mod $m$ is uniformly distributed. Then its characteristic function
satisfies $\phi_{X'}(\frac{2\pi a}{m})=0$ for $a=1, \dots, m-1.$
We also assume that
$\phi_{X'}'(\frac{2\pi a}{m})\neq 0$  for $a$ as above
 (for example one can suppose that
$X'$ takes the values $Lm$, $1, 2, \dots ,m-1$ with equal probabilities where $L$
is a large integer). Let $X''$ take values
 in $m\bbZ$  and have zero mean.  We also assume that $X''$ 
 does not take  values at $m_0\bbZ$ for a  larger $m_0$. Then $q(X'',m_0)>0$ for any $m_0\not=m$.
Fix $r\in \mathbb{N}$ and 
let $$X_n=\begin{cases} X' & n\leq r, \\ X'' & n>r. \end{cases}$$
Then $M_N(m_0)$ grows linearly fast in $N$ if $m_0\not=m$ and $M_N(m)$ is bounded in $N$. 
We claim that $S_N$ admits the Edgeworth expansion of order $r$
but does not admit Edgeworth expansion of order $r+1.$ 
The first statement holds due to Theorem \ref{r Char},
since $\Phi^{(\ell)}_N(\frac{2\pi a}{m})=0$ for each $a\in \bbZ$ and each
$\ell<r.$ On the other hand, since $\Phi_{N}^{(\ell)}(\frac{2\pi a}m)=0$ for any $\ell<r$, using Lemma \ref{Lem} we see that the conditions of Lemma \ref{LemInd} are satisfied with $r+1$ in place of $r$. Moreover, with $t_j=2\pi a/m$, $a\not=0$ we have $\cH_{N,r+1,s}(x,t_j)\equiv0$ for any $q\leq r+1$ and $s=2,3,4$ while $\cH_{N,q,w}(x,t_j)\equiv0$ for any $q\leq r$ and $w=1,2,3,4$.  Furthermore, when $N\geq r$ we have
\begin{eqnarray*}
\cH_{N,r+1,1}(x;t_j)=\frac{i^{r}H_{r}(x)\big(\phi_{X''}(2\pi a/m)\big)^{N-r}\Phi_{r}^{(r)}(2\pi a/m)}{r!}\\=(i)^{r}H_{r}(x)\big(\Phi_{X'}'(2\pi a/m)\big)^r.
\end{eqnarray*}
We conclude that 
$$ \bbP(S_N=k)$$
$$=\frac{e^{-k_N^2/2}}{\sqrt{2\pi}} \left[\cE_{r+1}(k_N)+
\frac{i^r}{\sig_N^{r+1}}
\sum_{a=1}^{m-1} e^{-2\pi i ak/m} \left(\phi_{X'}'\left(\frac{2\pi a}{m}\right)
\right)^r H_{r}(k_N) \right]
$$
$$
+o(\sig_N^{-r-1})
$$
where $\cE_{r+1}$ the Edgeworth polynomial (i.e. the contribution of 0)
and $H_{r}(x)$ is the Hermite polynomial.

Observe that since the uniform distribution on $\bbZ/m\bbZ$ is shift invariant,
$S_N$ are uniformly distributed mod $m$ for all $N\in \bbN.$ This shows that
for $r\geq 1$, one can not characterize Edgeworth expansions just in term of
the distributions of $S_N$ mod $m$, so the additional assumptions in Theorems
\ref{Thm SLLT VS Ege} and \ref{Thm Stable Cond} are necessary.

Next, consider a more general case where for each $n$, $X_n$ equals in law to either $X'$ or $X''$,
however, now we assume that $X'$ appears infinitely often. In this case $S_N$ obeys 
Edgeworth expansions of all orders since for large $N$, 
$\Phi_N(t)$ has zeroes of order greater  $N$ at all points of the form 
$\frac{2\pi a}{m},$ $ a=1, \dots, m-1.$
In fact, the Edgeworth expansions hold in the superstable way since removing a finite
number of terms does not make the order of zero to fall below $r.$

\end{example}

\begin{example}\label{Eg1}
Let $\fp_n=\min(1, \frac{\theta}{n})$ and 
let $X_n$ take value $0$ with probability $\fp_n$ and values $\pm 1$ with probability
$\frac{1-\fp_n}{2}.$ In this example the only non-zero resonant point is $\pi=2\pi\times \frac{1}{2}.$
Then 
for small $\theta$ the contributions of $P_{1, b, N}$ (the only non-zero $a$ is $1$)
are significant and as a result $S_N$ does not admit the ordinary Edgeworth expansion.
Increasing $\theta$ we can make $S_N$ to admit Edgeworth expansions of higher and higher 
orders. 
 Namely we get that for large $n,$
$\DS \phi_n(\pi)=\frac{2\theta}{n}-1. $ Accordingly
$$ \ln (-\phi_n(\pi))=-\frac{2\theta}{n}+O\left(\frac{1}{n^2}\right). $$
 Now the asymptotic relation 
$$ \sum_{n=1}^N \frac{1}{n}=\ln N+\fc+O\left(\frac{1}{N}\right),$$
where $\fc$ is the Euler-Mascheroni  constant,
 implies that that there is a constant $\Gamma(\theta)$ such that
$$ \Phi_N(\pi)=\frac{(-1)^N e^{\Gamma(\theta)}}{N^{2\theta}}\left(1+O(1/N)\right).$$  
Therefore $S_N$ admits the Edgeworth expansions of order $r$ iff 
$\DS \theta>\frac{r-1}{4}.$ Moreover, if
$\DS \theta\in \left(\frac{r-2}{4}, \frac{r-1}{4}\right],$ then 
Corollary \ref{CrFirstNonEdge}
shows that
$$ \mathbb{P}(S_N=k)=\frac{e^{-k_N^2/2}}{\sqrt{2\pi}} \left[\cE_{r}(k_N)+
\frac{(-1)^{N+k} e^{\Gamma(\theta)
}}{N^{2\theta+(1/2)}}+O\left(\frac{1}{N^{2\theta+1}
}\right)\right]
$$
where $\cE_r$ is the Edgeworth polynomial of order $r.$
In particular if $\theta\in (0, 1/4)$ then using that 
\begin{equation}
\label{VarNearN}
V_N=N+O(\ln N)=N\left(1+O\left(\frac{\ln N}{N}\right)\right)
\end{equation}
 and hence 
 \begin{equation}
 \label{SDNearN}
\sigma_N=\sqrt{N}\left(1+O\left(\frac{\ln N}{N}\right)\right) 
\end{equation}
we conclude that 
$$ \mathbb{P}(S_N=k)=\frac{e^{-k^2/(2N)}}{\sqrt{2\pi}} 
\left[\frac{1}{\sqrt{N}}+
\frac{(-1)^{N+k} e^{\Gamma(\theta)
}}{N^{2\theta+(1/2)}}+O\left(\frac{1}{N^{2\theta+1}}\right)\right].
$$

Next, take $\fp_n=\min\left(1, \frac{\theta}{n^2}\right)$. 
Then
the  SLLT does not hold, since the Prokhorov condition fails. 
Instead we have (\ref{r=1}) with $\cR=\{\pi\}$. Namely, uniformly in $k$ we have
\begin{equation*}
\sqrt 2\pi\bbP(S_N=k)=\left(1
+(-1)^k\prod_{u=1}^{N}\left(2\fp_u-1\right)\right)\sig_{N}^{-1}e^{-k^2/2V_N}+o(\sigma_N^{-1}).
\end{equation*}
Next, $\fp_u$ is summable and moreover
\[
\prod_{u=1}^{N}(2\fp_u-1)=(-1)^N U(1+O(1/N))
\]
where $\DS U=\prod_{n=1}^{\infty}(1-2\fp_u)$. We conclude that 
\begin{equation}
\label{A1Ex3}
\sqrt 2\pi\bbP(S_N=k)=\left(1+(-1)^{k+N}U\right)\sig_{N}^{-1}e^{-k^2/2V_N}+
O\left(\sigma_N^{-2}\right)
\end{equation}
uniformly in $k$. In this case the usual LLT holds true if and only if $U=0$
in agreement with Proposition \ref{PrLLT-SLLT}.

 In fact, in this case we have a faster rate of convergence.
To see this we consider expansions of order $2$ 
for $\fp_n$ as above. 
  We observe that $q_m(2)=\fp_n$ for large $n.$ Thus
  \[
|\bbE(e^{\pi iX_n})|=1-2 \fp_n
\] 
and so $|\bbE(e^{\pi X_n})|\geq \frac12$ when $n\geq N_\te$ for some minimal $N_\te$.
Therefore, we can take $N_0=N_\te$.
Note also that we have $Y_n=X_n\text{ mod }2-1$. We conclude that for $n>N_0$ we have
\[
a_{n}=a_{n,j}=\frac{\bbE[((-1)^{Y_n}-1)X_n]}{\bbE[(-1)^{Y_n}]}=0
\]
and so the term $C_{1,N}$ vanishes. Next, we observe that 
\[
\gamma_{N,j}=\frac{\sum_{n=N_0+1}^{N}\bbE(X_n^3)}{\sum_{n=N_0+1}^{N}(1-\fp_n)}=0.
\]
Finally, we note that $\bbE[(-1)^{X_n}X_n]=0$, and hence $\Phi_{N_0}'(\pi)=0$. Therefore, the second term in (\ref{r=2'}) vanishes and we have
\begin{equation*}
\sqrt 2\pi\bbP(S_N=k)=\left(1
+(-1)^{k+N} U
\right)\sig_{N}^{-1}e^{-k^2/(2V_N)}+
 O\left(\sigma_N^{-3}\right).
\end{equation*}
Taking into account \eqref{VarNearN} and \eqref{SDNearN} we obtain
$$ \sqrt 2\pi\bbP(S_N=k)=\frac{1
+(-1)^{k+N} U}{\sqrt{N}} \; e^{-k^2/(2N)}+
O\left(N^{-3/2}\right). $$
In particular, \eqref{A1Ex3} holds with the stronger rate $O\left(\sigma_N^{-3}\right)$.
\end{example}

\begin{example}\label{NonSym r=2 Eg}
 The last example exhibited significant simplifications. Namely, there was only one resonant
point, and, in addition, the second term vanished due to the symmetry.
We now show how a similar analysis could be performed when the above simplifications
are not present.
Let us assume that $X_n$ takes the values $-1,0$ and $3$ with probabilities $a_n,b_n$ and $c_n$ so that $a_n+b_n+c_n=1$. Let us also assume that
$ b_n<\frac18$ and that $a_n,c_n\geq \rho>0$ for some constant $\rho$. Then 
$$V(X_n)=9(c_n-c_n^2)+6a_nc_n+(a_n-a_n^2)\geq 6 \rho^2$$
\noindent
and so $V_N$ grows linearly fast in $N$.

Next, since we can take $K=3$, the denominators $m$ of the nonzero resonant points can only be $2,3,4,5$ or $6$. An easy check shows that for $m=3,5,6$ we have $q_n(m)\geq \rho$, and that for $m=2,4$ we have $q_n(m)=b_n$. Therefore, for $m=3,5,6$ we have $M_N(m)\geq \rho N$, and so we can disregard all the nonzero resonant points except for $\pi/2,\pi$ and $3\pi/2$. For the latter points we have
\begin{equation}
\label{Res1}
\phi_n\left(\frac{\pi}{2}\right)=b_n-i(1-b_n),
\end{equation}
\begin{equation}
\label{Res2}
 \phi_n(\pi)=2b_n-1,\quad 
\phi_n\left(\frac{3\pi}{2}\right)=b_n+i(1-b_n).
\end{equation}
Hence, denoting $\eta_n=b_n(1-b_n),$ we have
\begin{equation*}
\left|\phi_n\left(\frac{\pi}{2}\right)\right|^2=\left|\phi_n\left(\frac{3\pi}{2}\right)\right|^2=1-2\eta_n,
\quad
 \big|\phi_n(\pi)\big|^2=1-4\eta_n.
\end{equation*}
 Since we suppose that $\eta_n\leq b_n<\frac{1}{8}$ it follows that
$1-4\eta_n\geq \frac12$. Then for the above three resonant points we can take $N_0=0$. 
 Now Proposition~\ref{Thm} and a simple calculation show 
that
for any $r$ we get the Edgeworth expansions of order $r$ if and only if 
\[
\prod_{n=1}^N(1-2\eta_n)= o\left(N^{1-r}\right).
\]
Let us focus for the moment on the case when $b_n=\gamma/n$ for $n$ large enough
 where $\gamma>0$ is a constant. 
  Rewriting \eqref{Res1}, \eqref{Res2} as
\begin{equation}
\label{PhiNRatio}
 \frac{\phi_n\left(\frac{\pi}{2}\right)}{-i}=(1-b_n)+ib_n, \quad
\frac{\phi_n\left(\frac{3\pi}{2}\right)}{i}=(1-b_n)-ib_n, \quad
 \end{equation} 
 $$ -\phi_n(\pi)=1-2 b_n $$
 and,  using that the condition $b_n<\frac{1}{8}$ implies that 
 that $\phi_n(t)\neq 0$ for
 all $n\in \mathbb{N}$ and all $t\in \left\{\frac{\pi}{2}, \pi, \frac{3\pi}{2}\right\}$,
we conclude similarly to Example \ref{Eg1} that there are non-zero complex numbers
$\kappa_1, \kappa_3$ and a non-zero real number $\kappa_2$ such that
$$ \Phi_N\left(\frac{\pi}{2}\right)=\frac{(-i)^N \kappa_1}{N^\gamma} e^{i \gamma \ln N
}
\left(1+O\left(\frac1N\right)\right), $$
$$ \Phi_N\left(\frac{3\pi}{2}\right)=\frac{i^N \kappa_3}{N^\gamma} 
e^{-i \gamma \ln N
}
\left(1+O\left(\frac1N\right)\right), $$
$$ \Phi_N(\pi)=(-1)^N \frac{\kappa_2e^{2\gamma w_N}}{N^{2\gamma}}
\left(1+O\left(\frac1N\right)\right). $$
It follows that $S_N$ admits Edgeworth expansion of order $r$ iff
$\gamma>\frac{r-1}{2}.$
In fact if 
$\frac{r-1}{2}<\gamma\leq \frac{r}{2}$ then
Corollary \ref{CrFirstNonEdge}
shows that
$$ \mathbb{P}(S_N=k)=
\frac{e^{-k_N^2/2}}{\sqrt{2\pi}} \Bigg[\cE_{r}(k_N)+
\frac{\kappa_1  e^{i\gamma \ln N}}
{N^\gamma \sig_N}
+
\frac{\kappa_3  e^{-i\gamma \ln N}}
{N^\gamma \sig_N}
+O\left(N^{-\eta}\right)\Bigg]
$$
where $\cE_r$ is the Edgeworth polynomial of order $r$ and
$\eta\!\!=\!\!\min\left(2\gamma, \frac{r}{2}\right)+\frac{1}{2}.$

To give a specific example, let us suppose that $\frac{1}{2}\leq \gamma<1$ and that
$E(X_n)=0$ which means that
\begin{equation}
\label{Ex4-A-C}
a_n=\frac{3(1-b_n)}{4}, \quad c_n=\frac{1-b_n}{4}. 
\end{equation}
Then 
\begin{equation}
\label{Ex4M2M3}
 V_N=3N-3\gamma \ln N+O(1), \quad
E(S_N^3)=6N-6\gamma \ln N+O(1), 
\end{equation}
so Proposition \ref{2 Prop} gives
$$ \sqrt{2\pi} \mathbb{P}(S_N=k)=$$
$$e^{-k^2/6 N}
\left[\frac{1}{\sqrt{3N}}
\left(1+\frac{\kappa_1 i^{k-N} e^{i\gamma \ln N}
+\kappa_3 i^{N-k} e^{-i\gamma \ln N}}
{N^\gamma}\right)-
 \frac{k^3}{81\sqrt{3 N^5}} 
\right]
$$
$$
+O\left(N^{-3/2}\right)
$$


Next, let us provide the second order trigonometric expansions under the sole assumption that $1-4\eta_n\geq \frac12$ and $a_n,c_n\geq \rho$. As we have mentioned, we only need to consider the nonzero resonant points $\pi/2,\pi,3\pi/2$ and for these points we have $N_0=0$. Therefore, the term involving the derivative in the right hand side of (\ref{r=2'}) vanishes. 
Now, a direct calculation shows that 
\[
C_{1,N,\pi}=\sum_{n=1}^{N}\frac{\bbE(e^{i\pi X_n}\bar X_n)}{\bbE(e^{i\pi X_n})}=2\sum_{n=1}^{N}\frac{(a_n-3c_n)b_n}{2b_n-1}
\]
and 
$$
C_{1,N,\pi/2}=\sum_{n=1}^N\frac{(a_n-3c_n)(1+i)b_n}{b_n-i(1-b_n)},\quad
C_{1,N,3\pi/2}=\sum_{n=1}^N\frac{(a_n-3c_n)(1-i)b_n}{b_n+i(1-b_n)}.
$$
Note that $3c_n-a_n=\bbE(X_n)$.
Set
$$
\Gamma_{1,N}=\prod_{n=1}^{N}(b_n-i(1-b_n)), \quad 
\Gamma_{2,N}=\prod_{n=1}^N(2b_n-1), \quad
\Gamma_{3,N}=\prod_{n=1}^{N}(b_n+i(1-b_n)).
$$
Then $\Gamma_{s,N}=\bbE(e^{\frac{s\pi i}{2} S_N})$.
We also set 
$$
\Theta_{s,N}=C_{1,N,s\pi/2}\Gamma_{s,N},\,s=1,2,3
$$
and 
$$\Gamma_{N}(k)=\sum_{j=1}^{3}e^{-j\pi i k/2}\Gamma_{j,N},
\quad \Theta_{N}(k)=\sum_{j=1}^{3}e^{-j\pi i k/2}\Theta_{j,N}.$$
Then by Proposition \ref{2 Prop} and Remark \ref{Alter 2nd Order}, uniformly in $k$ we have
\begin{equation}\label{EgGen}
\sqrt{2\pi}P(S_N=k)=\sig_N^{-1}\left(1+\Gamma_N(k)\right)e^{-k_N^2/2}
\end{equation}
$$
-\sig_N^{-2}\left(k_N^3 T_N\big(1+\Gamma_{N}(k)\big)+ik_N\Theta_{N}(k)\right)e^{-k_N^2/2}+o(\sig_N^{-2})
$$
where $T_N=\frac{\DS \sum_{n=1}^{N} \bbE(\bar X_n^3)}{6V_N}$, $\bar X_n=X_n-\bbE(X_n)$.

Let us now consider a more specific situation. 
 Namely we suppose that  $b_n=\frac{\gamma}{n^{3/2}}$ for large $n$
and that $E(X_n)=0.$ Then \eqref{Ex4-A-C} shows that 
$C_{1, N, s\pi/2}=0.$ Next \eqref{PhiNRatio} gives
$$ \frac{\Phi_N\left(\frac{\pi}{2}\right)}{(-i)^N}= \prod_{n=1}^N \left[(1-b_n)+ib_n\right]=
\frac{\brkappa_1}{\DS \prod_{n=N+1}^\infty\left[(1-b_n)+ib_n\right]}$$
$$=
\brkappa_1 \left(1+\frac{2 \gamma(1-i)}{\sqrt{N}}+O\left(\frac{1}{N}\right)\right)
$$
where $\DS \brkappa_1=\prod_{n=1}^\infty \left[(1-b_n)+ib_n\right].$
Likewise
$$ \frac{\Phi_N\left(\frac{3\pi}{2}\right)}{i^N}=
\brkappa_3 \left(1+\frac{2 \gamma(1+i)}{\sqrt{N}}+O\left(\frac{1}{N}\right)\right)
$$
and
$$
 \frac{\Phi_N\left(\pi\right)}{(-1)^N}=
\brkappa_2 \left(1+\frac{4 \gamma}{\sqrt{N}}+O\left(\frac{1}{N}\right)\right).
$$
Taking into account \eqref{Ex4M2M3} we can reduce \eqref{EgGen} to the following 
expansion
$$ \sqrt{2\pi}\bbP(S_N=k)=e^{-k^2/6N} \left[
\frac{1}{\sqrt{3N}} \left(1+\sum_{s=1}^3 \brkappa_s i^{s(k-N)} \right)\right. $$
$$\left.+\frac{1}{N}\left(
-\frac{\tilde k_N^3}{3} 
+ \sum_{s=1}^3 \brkappa_s i^{s(k-N)} 
\left(\frac{2\gamma(1-i^{-s}) }{\sqrt{3}}-\frac{\tilde k_N^3}{3} \right)\right)
\right]+O\left(\frac{1}{N^{3/2}}\right) 
$$
where $\tilde k_N=k/\sqrt{3N}$.
\end{example}

\begin{example}
Let $X'$ take value $\pm 1$ with probability $\frac{1}{2}$, $X''$ take values $0$ and $1$ with probability $\frac{1}{2}$, and $X^{\delta}$, $\del\in[0,1]$ be the mixture of $X'$ and $X''$ with
weights $\delta$ and $1-\delta.$ Thus $X^\delta$ take value $-1$ with probability 
$\frac{\delta}{2},$ the value $0$ with probability $\frac{1-\delta}{2}$ and value $1$ with probability 
$\frac{1}{2}$. Therefore, $\bbE(e^{\pi i X^\delta})=-\del$.  We suppose that $X_{2m}$ and $X_{2m-1}$ have the same law which we
call $Y_m.$ The distribution of $Y_m$ is defined as follows. Set $k_j=3^{3^j}$, and let $Y_{k_j}$ have the same distribution as $X^{\del_j}$ where $\del_j=\frac{1}{\sqrt{k_{j+1}}}$. When $m\not\in\{k_j\}$ we 
let $Y_m$ have the distribution of $X'$. It is clear that $V_N$ grows linearly fast in $N$. Note also that  $\bbE(e^{\pi i Y_m})=-\del_j$ when $m=k_j$ for some $j$, and otherwise $\bbE(e^{\pi i Y_m})=-1$. 
Now, take $N\in\bbN$ such that $N>2k_2$, and let $J_N$ be so that $2k_{J_N}\leq N<2k_{J_N+1}$. Then 
$$
|\Phi_{N}(\pi)|\leq \prod_{j=1}^{J_N}(k_{j+1})^{-1}.
$$ 
Since $k_{J_N}\leq \frac{N}2<k_{J_N+1}$ and $k_j=(k_{j+1})^{1/3}$  we have 
$k_{J_N+1}^{-1}\leq 2N^{-1}$ 
and $k_{J_N+1-m}\leq 2^{3^{-m}} N^{-3^{-m}}$ for any $0<m\leq J_N$. 
 Denote $\DS \alpha_N=\sum_{j=1}^{J_N-1}3^{-j}.$
Since $\alpha_N>1/3$ we get that
\[
|\Phi_{N}(\pi)|\leq 2^{3/2}N^{-\alpha_N}
=o(N^{-1-1/3}).
\]
Similarly, for  each $j_1, j_2\leq N$, 
\begin{equation}\label{2nd}
|\Phi_{N: j_1}(\pi)|\leq  2^{3/2} N^{-1/2-\alpha_N}=o(N^{-1/2-1/3})
\end{equation}
and 
\begin{equation}\label{3rd}
|\Phi_{N: j_1,j_2}(\pi)|\leq 2^{3/2} N^{-\alpha_N}=o(N^{-1/3}).
\end{equation}
Indeed, the largest possible values are obtained for $j_1=2k_{J_N}$ (or $j_1=2k_{J_{N+1}}-1$ if it is smaller than $N+1$) and $j_2=2k_{J_N}-1$ (or $j_2=2k_{J_{N}}$).
Using the same estimates as in the proof of of Theorem \ref{Thm Stable Cond}  we conclude from (\ref{2nd})  that
$\DS \Phi_N'(\pi)=o\left(1/\sqrt{N}\right)$ and
we conclude from \eqref{3rd} that $\DS \Phi_N''(\pi)=o(1).$
It follows from Lemma \ref{Lem} and Proposition \ref{Thm}  that $S_N$ satisfies an Edgeworth expansion of order 3. The same conclusion holds if we remove a finite number of terms from the beginning of
the sequence $\{X_n\}$ because the smallness of $\Phi_N(\pi)$ comes from the terms
$X_{2 k_{j}-1} $ and $X_{2 k_j}$ for arbitrary large $j$'s.

On the other hand 
$$ \left|\Phi_{2k_j; 2k_j, 2k_j-1, 2k_{j-1}, 2k_{j-1}-1}(\pi)\right|=
\prod_{s=2}^{j-1} \left(3^{-3^{s}}\right)$$
$$=
 3^{-(3^j-9)/2}=\frac{3^{9/2}}{\sqrt{k_j}}
\gg \frac{1}{k_j}=3^{-3^j}. $$
It follows that $S_{2k_j; 2k_j, 2k_j-1, 2k_{j-1}, 2k_{j-1}-1}$ does not obey the Edgeworth
expansion of order 3. Accordingly, stable Edgeworth expansions need not be superstable
if $r=3.$ A similar argument allows to construct examples showing that those notions are different 
for all $r>2.$
 \end{example}

\section{Extension for uniformly bounded integer-valued triangular arrays}
 In this section we will describe our results for arrays of independent random variables.
We refer to \cite{Feller}, \cite{Mu84, Mu91} and \cite{Dub}, \cite{VS}, \cite{Pel1} and \cite{DS} for results for triangular arrays of inhomogeneous Markov chains. 
Example where Markov arrays appear naturally include the theory of large deviations
for inhomogeneous systems (see \cite{SaSt91, PR08, FH} and references wherein),
random walks in random scenery \cite{CGPPS, GW17}, and statistical mechanics
\cite{LPRS}.

Let $X_n^{(N)},\,1\leq n\leq L_N$ be a triangular array such that for each  fixed $N$, the random variables $X_n^{(N)}$ are independent and integer valued. Moreover, we assume that 
$$K:=\sup_{N}\sup_{n}\|X_n\|_{L^\infty}<\infty.$$ 
For each $N$ we set $\DS S_N=\sum_{n=1}^{L_N}X_n^{(N)}$. Let $V_N=\text{Var}(S_N)$. 
We assume that $V_N\to\infty$, so that, by Lindenberg--Feller Theorem,
the sequence $(S_N-\bbE(S_N))/\sig_N$ obeys the CLT, where $\sig_N=\sqrt{V_N}$. 
 
 We say that the array  $X_n^{(N)}$ obeys the SLLT  if for any $k$ the LLT holds true for any uniformly square integrable array $Y_n^{(N)},\,1\leq n\leq L_N$, so that $Y_n^{(N)}=X_n^{(N)}$ for all but $k$ indexes $n$.
Set
$$
M_N:=\min_{2\leq h\leq 2K}\sum_{n=1}^{L_N}P(X_n\neq m_n^{(N)}(h) \text{ mod } h)\geq R\ln V_N
$$
where $m_n^{(N)}(h)$ is the most likely value of $X_n^{(N)}$ modulo $h$.
Observe now that  the proofs of Proposition \ref{PropEdg} and Lemmas  \ref{Step1}, \ref{Step2}, \ref{Step3} and \ref{Step4} proceed exactly the same for arrays. Therefore, all the arguments in the proof of Theorem \ref{IntIndThm} proceed the same for arrays instead of a fixed sequence $X_n$. That is, we have
\begin{theorem}\label{IntIndThmAr}
There $\exists J=J(K)<\infty$ and polynomials $P_{a, b, N}$ with degrees depending only on $a$ and $b$, whose coefficients are uniformly bounded in $N$ such that, for any $r\geq1$ uniformly in $k\in\bbZ$ we have
$$\bbP(S_N=k)-\sum_{a=0}^{J-1} \sum_{b=1}^r \frac{P_{a, b, N} ((k-a_N)/\sigma_N)}{\sigma_N^b}
\fg((k-a_N)/\sigma_N) e^{2\pi i a k/J} =o(\sigma_N^{-r})
$$
where $a_N=\bbE(S_N)$ and $\fg(u)=\frac{1}{\sqrt{2\pi}} e^{-u^2/2}. $

Moreover, $P_{0,1,N}\equiv1$ and 
given $K, r$, there exists $R=R(K,r)$ such that if 
$M_N\geq R \ln V_N$ then we can choose $P_{a, b, N}=0$ for $a\neq 0.$ 
\end{theorem}
All the formulas for the coefficients of the polynomials $P_{a,b,N}$ remain the same in the arrays setup.
 In particular, we get that, uniformly in $k$ we have
\begin{equation}\label{r=1 Ar}
\bbP(S_N=k)=\left(1+\sum_{t\in\cR}e^{-itk}\Phi_N(t)\right)e^{-k_N^2/2}\sig_N^{-1}+o(\sig_N^{-1})
\end{equation}
where $\Phi_N(t)=\bbE(e^{it S_N})$.

Next, our version for Proposition \ref{PrLLT-SLLT} for arrays is as follows.
\begin{proposition}
\label{PrLLT-SLLT Ar}
Suppose $S_N$ obeys LLT. Then for each  integer $h\geq2$,  at least one of the following conditions occur:
\vskip0.2cm
either (a) $\DS \lim_{N\to\infty}\sum_{n=1}^{L_{N}} \bbP(X_n\neq m_n^{(N)}(h) \text{ mod } h)=\infty$.
\vskip0.2cm
or (b) there exists a subsequence $N_k$, numbers $s\in\bbN$ and $\ve_0>0$ and  indexes $1\leq j_1^{k},...,j_{s_k}^{k}\leq L_{N_k}$, $s_k\leq s$  so that the distribution of  
$\DS \sum_{u=1}^{s_k}X^{(N_k)}_{j_u}$ converges to uniform $\text{mod }h$, and the distance between the distribution of $\DS S_{N_k}-\sum_{q=1}^{s_k}X^{(N_k)}_{j_q}$ and the uniform distribution $\text{mod }h$ is at least $\ve_0$.
\end{proposition}

\begin{proof}
First, by \eqref{r=1 Ar} and Lemma \ref{Lemma} if the LLT holds 
then for any nonzero resonant point $t$ we have $\DS \lim_{N\to\infty}|\Phi_N(t)|=0$. Now, if (a) does not hold true then there is a subsequence $N_k$ so that 
$\DS \sum_{n=1}^{L_{N_k}}q(X_n^{(N_k)},h)\leq C$, where $C$ is some constant. Set 
$\DS q_n^{(N_k)}(h)=q\left(X_n^{(N_k)},h\right)$. Then there are at most $8hC$ $n$'s  between $1$ and $L_{N_k}$ so that $q_n^{(N_k)}(h)>\frac1{8 h}$. Let us denote these $n$'s by $n_{1,k},...,n_{s_k,k}$, $s_k\leq 8h C$. Next, for any $n$ and a nonzero resonant point $t=2\pi l/h$ we have 
\begin{equation}
\label{PhiQArr}
  |\phi_n^{(N_k)}(t)|\geq 1-2hq_n^{(N_k)}(h) 
 \geq e^{-2\gamma h q_n^{(N_k)}}
\end{equation}
where $\phi_n^{(N_k)}$ is the characteristic function of $X_n^{(N_k)}$
 and $\gamma$ is such that for $\theta\in [0, 1/4]$ we have
$1-\theta\geq e^{-\gamma \theta}.$
 We thus get that 
\begin{equation}\label{C1.}
\prod_{n\not\in\{n_{u,k}\}}|\phi_{n}^{(N_k)}(t)|\geq\prod_{n\not\in\{n_{u,k}\}}(1-2hq_n^{(N_k)}(h))\geq C_0
\end{equation}
where $C_0>0$ is some constant.
Therefore,
$$
|\Phi_{N_k}(t)|\geq \prod_{u=1}^{s_k}|\phi_{n_{u,k}}^{(N_k)}(t)|\cdot C_0
$$
and so we must have
\begin{equation}\label{C2.}
\lim_{k\to\infty}\prod_{u=1}^{s_k}|\phi_{n_{u,k}}^{(N_k)}(t)|=0.
\end{equation}
Now (b) follows from \eqref{C1.}, \eqref{C2.} and Lemma \ref{LmUnifFourier}.
\end{proof}

Using (\ref{r=1 Ar})   we can now prove a version of Theorem \ref{ThProkhorov} for arrays.
\begin{theorem}\label{ThProkhorovAr}
The SLLT holds iff
for each  integer $h>1$,
\begin{equation}\label{Prokhorov Ar}
\lim_{N\to\infty}\sum_{n=1}^{L_N} \bbP(X_n^{(N)}\neq m_n \text{ mod } h)=\infty 
\end{equation}
where $m_n=m_n^{(N)}(h)$ is the most likely residue of $X_n^{(N)}$ modulo $h$.
\end{theorem}

\begin{proof}
First, the arguments in the proof of (\ref{Roz0}) show that there are constants $c_0,C>0$ so that
for any nonzero resonant point $t=2\pi l/h$ we have
\begin{equation}\label{ROZ}
|\Phi_N(t)|\leq Ce^{-c_0M_N(h)},\quad \text{where}\quad M_N(h):=\sum_{n=1}^{L_N}q(X_n^{(N)},h).
\end{equation}
Let us assume that \eqref{Prokhorov Ar} holds 
for all integers $h>1$.  Consider $s_N$--tuples
 $1\leq j_1^N,...,j_{s_N}^N\leq L_N$, where $s_N\leq \bar s$ is bounded in $N$. Then by applying \eqref{ROZ} with $\DS \tilde S_N=S_N-\sum_{l=1}^{s_N}X_{j_l^N}^{(N)}$ we have 
\begin{equation}\label{TILD}
\lim_{N\to\infty}|\bbE(e^{it\tilde S_N})|=0.
\end{equation}
Now, 
arguing as in the proof of  Theorem \ref{Thm Stable Cond}(1), given a uniformly square integrable array $Y_n^{(N)}$ as in the definition of the SLLT, we still have \eqref{r=1 Ar}, even though the new array is not necessarily uniformly bounded. Applying (\ref{TILD}) we see that for any nonzero resonant point $t$ we have 
$$
\lim_{N\to\infty}\left|\bbE\left(\exp\left[it \sum_{n=1}^{L_N}Y_n^{(N)}\right]\right)\right|=0
$$
and so $\DS S_N Y:=\sum_{n=1}^{L_N}Y_n^{(N)}$ satisfies the LLT.

Now let us assume that $M_N(h)\not\to\infty$  for some $2\leq h\leq 2K$ 
(it not difficult to see 
 that \eqref{Prokhorov Ar} holds for any $h>2K$). 

In other words after taking a subsequence we have that  
$M_{N_k}(h)\leq L$ for some $L<\infty.$  
 The proof of Proposition \ref{PrLLT-SLLT Ar} 
 shows  that 
 there $s<\infty$ such that 
 after possibly removing terms $n_{1, k}, n_{2, k},\dots ,n_{s_k, k}$ with $s_k\leq s$ 
 we can obtain that $q_n^{(N_k)}(h) \leq \frac{1}{8h},$ $n\not\in\{n_{j, k}\}$.
In this case \eqref{PhiQArr} shows that for each $\ell$
$$ |\Phi_{N_k; n_{1, k} , \dots,  n_{s_k, k}}(2\pi\ell/h)|\geq e^{-2\gamma L}. $$
By Proposition \ref{PrLLT-SLLT Ar},
$S_{N_k; n_{1, k} , \dots,  n_{s_k, k}}$ does not satisfy the LLT.
\end{proof}

Next, all the other arguments in our paper  proceed similarly for arrays since they essentially rely only on the specific structure of the polynomials from Theorem \ref{IntIndThm}.
 For the sake of completeness, let us formulate the main (remaining) results here.

\begin{theorem}\label{ThLLT Ar}
The following conditions are equivalent:

(a) $S_N$ satisfies LLT;

(b) For each $\xi\in \bbR\setminus \bbZ$, 
$\DS \lim_{N\to\infty} \bbE\left(e^{2\pi i \xi S_N}\right)=0$;

(c) For each non-zero resonant point $\xi$,
$\DS \lim_{N\to\infty} \bbE\left(e^{2\pi i \xi S_N}\right)=0$;

(d) For each integer $h$ the distribution of $S_N$ mod $h$ converges to uniform.
\end{theorem}

\begin{theorem}
\label{ThEdgeMN Ar}
For each $r$ there is $R=R(r, K)$  such that 
the Edgeworth expansion of order $r$ holds true if $M_N\geq R\ln V_N$. In particular, $S_N$ obeys Edgeworth expansions of all orders if
$$ \lim_{N\to\infty} \frac{M_N}{\ln V_N}=\infty. $$
\end{theorem}

\begin{theorem}\label{r Char Ar}
 For any $r\geq1$, the Edgeworth expansion of order $r$ holds if and only if for any nonzero resonant point $t$ and $0\leq\ell<r$ 
we have
\[
\bar \Phi_{N}^{(\ell)}(t)=o\left(\sig_N^{\ell+1-r}\right)
\]
where $\bar\Phi_{N}(x)=\bbE[e^{ix (S_N-\bbE(S_N))}]$.
\end{theorem}

\begin{theorem}\label{Thm SLLT VS Ege Ar }
Suppose $S_N$ obeys the SLLT. Then  the following are equivalent:

(a) Edgeworth expansion of order 2 holds;

(b) $|\Phi_N(t)|=o(\sigma_N^{-1})$ for each nonzero resonant point $t$;

(c) For each $h\leq 2K$ the distribution of $S_N$ mod $h$ is
$o(\sigma_N^{-1})$ close to uniform.
\end{theorem}

Next, we say that  an array $\{X_n^{(N)}\}$ 
{\em admits an Edgeworth expansion of order $r$ in a superstable way} 
(denoted by $\{X_n^{(N)}\}\in EeSs(r)$) if for each $\brs$  and each
sequence $j_1^N, j_2^N,\dots ,j_{s_N}^N$ with $s_N\leq \brs$ and $j_i^N\leq L_N$
there are polynomials $P_{b, N}$ whose coefficients are $O(1)$
in $N$ and their degrees do not depend on $N$ so that 
 uniformly in $k\in\bbZ$ we have
that
\begin{equation}\label{EdgeDefSS Ar}
\bbP(S_{N; j_1^N, j_2^N, \dots,j_{{s_N}^N}}=k)=\sum_{b=1}^r \frac{P_{b, N} (k_N)}{\sigma_N^b}
\fg(k_N)+o(\sigma_N^{-r})
\end{equation}
and the estimates in $O(1)$ and $o(\sigma_N^{-r})$ are uniform in the choice
of the tuples $j_1^N, \dots ,j_{s_N}^N.$ 

Let $\Phi_{N; j_1, j_2,\dots, j_s}(t)$ be the characteristic function of
$S_{N; j_1, j_2,\dots, j_s}.$

\begin{theorem}\label{Thm Stable Cond Ar}
(1) $S_N\in EeSs(1)$ (that is, $S_N$ satisfies the LLT 
in a superstable way) if and if it satisfies the SLLT.

(2) For arbitrary $r\geq 1$ the following conditions are equivalent:

(a) $\{X_n^{(N)}\}\in EeSs(r)$;

(b) For each $j_1^N, j_2^N,\dots ,j_{s_N}^N$ and each nonzero resonant point $t$ we have
$\Phi_{N; j_1^N, j_2^N,\dots, j_{s_N}^N}(t)=o(\sigma_N^{1-r});$

(c) For each $j_1^N, j_2^N,\dots ,j_{s_N}^N$, and each $h\leq  2K$ 
the distribution of $S_{N; j_1^N, j_2^N,\dots, j_{s_N}^N}$ mod $h$ is
$o(\sigma_N^{1-r})$ close to uniform.
\end{theorem}

\end{document}